\documentclass[a4paper]{article}
\usepackage{mathrsfs}
\usepackage{latexsym,bm}
\usepackage{amssymb,amsmath,amsthm}
\usepackage[english]{babel}
\usepackage[colorlinks=true]{hyperref}
\usepackage{graphicx}
\usepackage{color}
\usepackage[title]{appendix}
\usepackage{titlesec}
\usepackage{thmtools}  

\allowdisplaybreaks[4]

\newtheorem{theorem}{Theorem}[section]
\newtheorem{lemma}[theorem]{Lemma}
\newtheorem{proposition}[theorem]{Proposition}

\declaretheorem[
name=Remark,        
style=definition,
sibling=theorem
]{remark}
\newtheorem{corollary}[theorem]{Corollary}
\numberwithin{equation}{section}

\hypersetup{linkcolor=blue,urlcolor=red,citecolor=red}

\title{Quantitative rapid boundary stabilization via modal decomposition and its application to the Allen--Cahn equation}
\author{Shengquan Xiang{\footnote{School of Mathematical Sciences, Peking University, Beijing 100871, China (e-mail: shengquan.xiang@math.pku.edu.cn).}}\and Yu Xiao{\footnote{School of Mathematics and Statistics, Wuhan University, Wuhan 430072, China (e-mail: xiaoyu\_math@whu.edu.cn).}}\and Can Zhang{\footnote{School of Mathematics and Statistics, Wuhan University, Wuhan 430072, China (e-mail: canzhang@whu.edu.cn).}}}
\date{}

\begin{document}
\selectlanguage{english}
\maketitle

\begin{abstract}
We investigate quantitative rapid stabilization for the one-dimensional Allen--Cahn equation and develop a quantitative modal decomposition approach that makes explicit the dependence of the feedback laws and stabilization costs on the prescribed decay rate. We construct an explicit feedback law on the finite-dimensional unstable modes via Ackermann's formula. The explicit structure of the feedback allows us to derive quantitative low-frequency estimates, which, combined with the frequency Lyapunov method, yield quantitative stabilization estimates. Together with the stabilization framework of \cite{XZ}, the resulting estimates can be adapted to a broader class of one-dimensional parabolic models. We further construct piecewise feedback laws that yield the null controllability with control costs and finite-time stabilization.
\end{abstract}

{\bf Keywords.} Modal decomposition, frequency Lyapunov, quantitative stabilization, controllability
\vskip 5pt
{\bf AMS subject classifications.} 35K55, 93B05, 93D15

\section{Introduction}
\subsection{Motivation}
\noindent \textit{Modal decomposition and rapid stabilization.} Stability analysis, particularly boundary feedback stabilization, is a core topic in the control theory of parabolic equations. The goal of exponential stabilization is to design
a feedback law ensuring that the energy $E(t)$ of the closed-loop system decays exponentially at a prescribed
rate $\lambda>0$: 
\begin{equation*}
E(t) \leq M_{\lambda}\mathrm{e}^{-\lambda t}E(0),\;\; \forall t\ge 0,
\end{equation*} 
for some $M_\lambda>0$. If such a feedback law can be constructed for arbitrarily large $\lambda$, we call this \textit{rapid stabilization}.
\vspace{1mm}

We recall the significant role played by linear--quadratic (LQ) approaches in the stabilization of infinite-dimensional systems. Developed in the works of Lasiecka--Triggiani \cite{LT}, Barbu--Triggiani \cite{BT}, Barbu--Wang \cite{BW}, and others, LQ theory provides a systematic Riccati-equation approach for constructing stabilizing feedback laws from suitable optimal control problems. Within this framework, the authors of \cite{TWX} established an equivalence between stabilizability and weak observability inequalities, which was applied to parabolic systems \cite{HWW}. For evolution equations governed by $C_0$-groups, the Gramian method plays a central role in achieving rapid stabilization, as shown in \cite{Ur}; it can be viewed from an optimal-control perspective.

In fact, stabilization is often realized by modifying the \textit{spectrum} of the system, and spectral decomposition methods \cite{BT,CT1,X} have gained considerable popularity. The core idea of these approaches is to decompose the equation into \textit{low- and high-frequency} components. Utilizing the intrinsic dissipativity of the high-frequency part, the problem is reduced to stabilizing a finite-dimensional low-frequency subsystem.
\vspace{1mm}

The \textit{modal decomposition} method is a powerful technique among spectral decomposition methods. By spectral analysis and state decomposition, this method constructs a state-feedback control using only a finite number of \textit{Fourier modes}. In recent years, observer-based control has attracted growing interest. Notably, modal decomposition has evolved into a widely adopted technique for implementing observer-based control for linear systems; see, e.g., Katz-Fridman \cite{KF} and Lhachemi-Prieur \cite{LP1}. Since only a finite set of easily accessible modal information is required, this method provides an effective implementation framework for both state-feedback and observer-based control schemes.
\vspace{1mm}

The stabilization of nonlinear systems is generally more challenging than that of linear systems. However, one often expects a local robustness property: a feedback stabilizing the
linearized counterpart may also stabilize the nonlinear system under suitable smallness assumptions. Treating nonlinear terms as perturbations, Lyapunov-based stability analysis provides an effective approach for stabilizing nonlinear systems; its core is the construction of an exponentially stable Lyapunov function.

In the context of addressing the controllability of semilinear heat equations, Coron-Tr\'{e}lat \cite{CT1} combined the modal decomposition method with a constructive Lyapunov function to establish the stability on a finite time interval. On the basis of this stability result, they further derived the global controllability between steady states. In recent years, the Coron-Tr\'{e}lat framework has been widely applied to address various stabilization problems of one-dimensional parabolic equations. These include linear, nonlinear, stochastic, cascaded, and state-delay problems, as well as time-delay and sampled-data control systems; see, e.g., \cite{KF,KF2,LP1,LP4,LS,PWF,PT,SWZ}.\\[-2mm]
 
\noindent \textit{Frequency Lyapunov method and quantitative rapid stabilization.} In rapid stabilization, the constant $M_\lambda$ represents the stabilization cost associated with the prescribed decay rate $\lambda$. The problem of \textit{quantitative rapid stabilization} is to obtain a quantitative estimate of $M_\lambda$ in terms of $\lambda$.

Such estimates are important for understanding the behavior of the closed-loop system. Indeed, $M_\lambda$ affects both the robustness of the closed-loop system and the time needed for the energy to reach a prescribed level. For instance, under a bounded perturbation of size $\varepsilon$, one typically obtains an estimate of the form $E(t)\leq M_\lambda e^{-(\lambda-C_\varepsilon M_\lambda)t}E(0)$. Thus, the growth of $M_\lambda$ directly influences the admissible perturbation size. Moreover, $M_\lambda$ also determines how fast the energy can reach a prescribed threshold. If $M_\lambda=e^{\mathcal O(\lambda^\alpha)}$ with $\alpha>1$, then the decay factor $e^{-\lambda t}$ cannot compensate for this cost within a short time, and a much longer stabilization time may be required.
\vspace{1mm}

In recent years, quantitative rapid stabilization has attracted growing interest. One important method is backstepping, introduced by Krsti\'c and his collaborators \cite{BK,KS}. It constructs a feedback law through a Volterra transformation so as to shift the overall spectrum of the closed-loop system; see also the Fredholm backstepping method in Coron--L\"u \cite{CL1}. For quantitative results in the Volterra and in particular the recent achievement in Fredholm backstepping frameworks, we refer to \cite{CN} and \cite{GHMXZ}, respectively. We also mention our recent work \cite{XXZ}, where the quantitative rapid stabilization of parabolic equations was established through LQ theory, and the equivalence between quantitative rapid stabilization and quantitative observability inequalities was proved for linear systems.

Another important approach is the \textit{frequency Lyapunov} method introduced by Xiang \cite{X}. This method separates the dynamics into low and high frequencies: the high-frequency part is controlled by intrinsic dissipation, while the low-frequency part is stabilized by a finite-dimensional feedback law. The Lyapunov function is then designed to balance these two mechanisms. The frequency Lyapunov method has been applied to several nonlinear stabilization problems, including the Navier--Stokes equations \cite{X1} and the heat flow \cite{CX1}, and has also inspired related developments for stochastic problems \cite{HLP}, disturbance problems \cite{GPC} and boundary control \cite{XZ}.
\vspace{1mm}

The quantitative viewpoint further provides constructive controllability results. Explicit bounds on $M_\lambda$ can be used in time-iteration arguments to derive the finite-time null controllability with explicit control costs, and can also be used to finite-time stabilization, as introduced in \cite{CN}. This differs from the classical Carleman-based approach to nonlinear null controllability \cite{PWZ}, and from the Lebeau--Robbiano strategy \cite{LR,M}, where control costs are obtained through spectral inequalities. In the quantitative stabilization strategy, the control cost is obtained directly from the feedback estimates on each time interval; see \cite{X2,X1,X}. Related ideas have also been used for the controllability of stochastic heat equations \cite{HLP}.
\\[-2mm]

In a recent work by the last two authors \cite{XZ}, the frequency Lyapunov method was combined with modal decomposition to achieve rapid boundary stabilization for one-dimensional nonlinear parabolic equations. However, to the best of our knowledge, few works have provided \textit{quantitative} results within the framework of modal decomposition. The purpose of this paper is to further investigate the quantitative properties of the frequency Lyapunov method applied to the modal-decomposition feedback proposed in \cite{XZ}. The main challenge consists in deriving quantitative low-frequency estimates.\\[-2mm]

Let $s\ge 0$. We focus on the stabilization of the Allen--Cahn model
\begin{equation}\label{h-1-first}
y_t(t,x)=y_{xx}(t,x)+y(t,x)-y^3(t,x)\;\;\;\; \quad\text{for }\;(t,x)\in (s,+\infty)\times (0,1).
\end{equation}
To stabilize \eqref{h-1-first}, we first suppose that the control acts on the left boundary as $y(t,0)=a(t)$. In order to address regularity issues and motivated by the control constructed in \cite{X2,XZ}, we introduce an auxiliary scalar control $u(\cdot)$ to stabilize the system in higher-regularity spaces:
\begin{equation}\label{a67}
a_t(t)=(1-\frac{\pi^2}{4})a(t)+u(t)\quad\text{for }\;\;t\in(s,+\infty).
\end{equation}

We introduce the Banach space $\mathcal{H}=L^2(0,1)\times \mathbb{R}$ with the norm $\|(y,a)\|_{\mathcal{H}}=\|y\|_{L^2(0,1)}+|a|$. Regarding $(y,a)$ as the state, we obtain the following coupled system:
\begin{equation}\label{heat-sys}
\left\{\begin{array}{ll}
 y_t(t,x) = y_{xx}( t,x)+y(t,x)-y^3( t,x)&\text{for }(t,x)\in(s,+\infty)  \times (0,1),\\[2mm]
y(t,0)=a(t),\;\;\;y(t,1)=0&\text{for }t\in(s,+\infty) ,\\[2mm]
a_t(t)=(1-\frac{\pi^2}{4})a(t)+u(t)&\text{for }t\in(s,+\infty),
\end{array}\right.
\end{equation}
with the initial datum $(y_0,a_0)\in \mathcal H$, where the control $u(\cdot)$ takes values in $\mathbb{R}$.

\subsection{Main result}
The main result of this paper is the quantitative rapid stabilization within the framework of modal decomposition.
 \begin{theorem}[Quantitative rapid stabilization]\label{qu-ra-sta}
There exists $D\ge1$ such that, for every $\lambda>2\pi^2$, there exists a linear feedback law $\mathcal{K}_{\lambda}:\mathcal{H}\to\mathbb{R}$ depending only on the finite-dimensional low-frequency modes, with the following property: for every $s\ge 0$ and for every $(y_0,a_0)\in \mathcal{H}$ satisfying $\|(y_0,a_0)\|_{\mathcal{H}}\le \rho_{\lambda}:=e^{-D\sqrt{\lambda}\ln\lambda}$, the system \eqref{heat-sys} with $u(t)=\mathcal{K}_{\lambda}(y(t),a(t))$ admits a unique solution $(y,a)\in C([s,+\infty);\mathcal{H})$ satisfying 
\begin{equation}\label{exp-s-heat}
|u(t)|+  \|(y(t),a(t))\|_{\mathcal{H}}\le e^{D\sqrt{\lambda}\ln\lambda}e^{-\lambda (t-s)}\|(y_0,a_0)\|_{\mathcal{H}},\;\;\forall t\ge s.
\end{equation}
\end{theorem}

To the best of our knowledge, this is the first quantitative result on rapid stabilization within the modal-decomposition framework. Moreover, this quantitative aspect also allows us to explore further properties such as controllability, disturbance, observer,  and stochastic versions. The present work has two main {novelties}:
\begin{itemize}
	\item[1.] Traditional modal decomposition arguments usually give qualitative rapid stabilization: the finite-dimensional low-frequency subsystem is stabilized by pole placement, while the high-frequency component is handled by dissipation. However, the pole-placement theorem only guarantees the existence of controller gains and Lyapunov matrices. In this paper, we overcome this difficulty by using explicit expressions for the feedback gains via Ackermann's formula \eqref{ac-f}. The explicit construction allows us to obtain quantitative low-frequency estimates. More precisely, by assigning the eigenvalues in the pole-placement procedure and using suitable matrix-estimation techniques, we reduce the quantitative analysis to estimating the distances between the pole-placement eigenvalues and the low-frequency eigenvalues of the linearized counterpart.
	
	\item[2.] Beyond the derivation of explicit estimates, our analysis provides a quantitative realization of the frequency Lyapunov method in the context of boundary rapid stabilization. In particular, it extends the qualitative boundary framework  in \cite{XZ} to a quantitative setting for one-dimensional nonlinear parabolic equations such as the Allen--Cahn equation.
\end{itemize}

\begin{remark}\label{remark2-control}
Several comments are in order.
\begin{itemize}
\item[1.] The main point of Theorem \ref{qu-ra-sta} is not only the existence of an arbitrarily large decay rate, but the quantitative estimate of the corresponding stabilization cost and admissible local radius. The factor $e^{D\sqrt{\lambda}\ln\lambda}$ may also be expressed, after enlarging the constant, in the form $e^{D_p\lambda^{\frac{p}{p+1}}}$ for any $p>1$; see Remark \ref{quant-re} for a detailed discussion. This form is especially convenient in the time-iteration arguments leading to null controllability.

\item[2.] Since the modal decomposition method serves as an important tool for addressing boundary stabilization problems of one-dimensional parabolic models, the quantitative framework established in this paper can be applied to various settings. For instance, it can be extended to stabilization problems involving state-delay equations, observer-based inputs, time-delay inputs/outputs, sampled-data inputs, and stochastic differential equations. In addition, modal decomposition has also been adopted to achieve stabilization for other PDEs beyond parabolic equations, suggesting that the quantitative framework proposed in this paper is also promising for such applications.

\item[3.] Combined with piecewise feedback laws, the quantitative methodology yields a constructive derivation of the boundary null controllability with explicit control costs, thereby going beyond previous qualitative results based on Carleman estimates and fixed-point arguments. We will demonstrate this approach on the Allen--Cahn equation as a representative model in the sequel; building on \cite{X,XZ}, the quantitative estimates addressed here illustrate how the frequency Lyapunov method can be quantitatively implemented for boundary rapid stabilization and boundary null controllability in concrete one-dimensional nonlinear parabolic models.
\end{itemize} 
\end{remark}

As a consequence of the quantitative rapid stabilization, we obtain null controllability and finite time stabilization results.
\begin{corollary}[Null controllability with explicit control costs]\label{main-th-null-control}
Let $T \in (0,1) $. Then, for any $p>1$, there exist ${Q}_p \geq 1$ and a piecewise feedback law $U_p:(0,T)\times \mathcal{H}\to \mathbb{R}$ such that, for every $(y_0,a_0)\in\mathcal{H}$ satisfying $\|(y_0,a_0)\|_{\mathcal{H}}\leq R_p:=e^{-{{Q}_p}/{T^p}}$, the solution of the system \eqref{heat-sys} with $u(t)=U_p(t,(y(t),a(t)))$ satisfies $(y(T),a(T))=0$ in $\mathcal{H}$. Moreover, the control cost is given by 
\begin{equation}\label{new-2-1}
 \|u\|_{L^\infty(0,T)}\le e^{{{Q}_p}/{T^p}} \|(y_0,a_0)\|_{\mathcal{H}}.
\end{equation}
\end{corollary}

\begin{corollary}[Finite-time stabilization] \label{f-t-s}
Let $T\in(0,1)$. There exist $\varrho>0$ and an explicit $T$-periodic piecewise feedback law\footnote{A $T$-periodic piecewise feedback law is a piecewise feedback law $U$ satisfying $U(t+T,(y,a))=U(t,(y,a)),\;\forall t\in\mathbb{R}$.} $U:\mathbb{R}\times\mathcal{H}\to\mathbb{R}$ such that the following properties hold.

\smallskip

\noindent\textup{(i) Well-posedness.}
The feedback law $U$ is proper. More precisely, for every initial time $s\ge 0$ and every initial datum $(y_0,a_0)\in\mathcal{H}$, the closed-loop system \eqref{heat-sys} with $u(t)=U(t,(y(t),a(t)))$ admits a unique global mild solution
$(y,a)\in C([s,+\infty);\mathcal{H})$. We denote the corresponding closed-loop solution map by
\[
\Upsilon(t,s;(y_0,a_0)) := (y(t),a(t)),\quad t\ge s.
\]

\smallskip

\noindent\textup{(ii) Finite-time null controllability.}
If $\|(y_0,a_0)\|_{\mathcal{H}}\le \varrho$, then
\[
\Upsilon(s+2T,s;(y_0,a_0))=0,\quad \forall\, s\ge 0.
\]

\smallskip

\noindent\textup{(iii) Uniform stability.}
For each $\epsilon>0$, there exists $\delta >0$ such that, for every $t'\ge 0$ and $\|(y_0,a_0)\|_{\mathcal{H}}\le \delta$, we have $$\| \Upsilon(t,t';(y_0,a_0))\|_{\mathcal{H}}\le\epsilon, \;\;\forall t\ge t'.$$
\end{corollary}

\noindent \textbf{Organization of the paper.} The rest of the paper is organized as follows. In Section \ref{control}, based on the standard modal decomposition method, a finite-dimensional low-frequency feedback control is designed. Subsequently, Section \ref{estimate} carries out quantitative low-frequency estimates, which constitute the core of this work. Finally, by virtue of the frequency Lyapunov method, Section \ref{Stabilizatipn} presents the proof of Theorem \ref{qu-ra-sta}, namely the proof of quantitative rapid stabilization. The appendices are devoted to several well-posedness results, as well as to the proofs of Corollaries \ref{main-th-null-control} and \ref{f-t-s}.\\[-2mm]

\noindent \textbf{Notation.} Denote by $\langle\cdot,\cdot\rangle$ the inner product in $L^2(0,1)$, by $|\cdot|_2$ the Euclidean norm on $\mathbb{R}^n$, and by $\|\cdot\|_2$ the induced matrix norm. A nonnegative definite matrix $P$ is written as $P\succeq 0$. For positive constants $F_1,F_2$, we write $F_1=e^{\mathcal{O}(F_2)}$ if $F_1\le e^{CF_2}$ for some $C>0$ independent of $F_1,F_2$. Furthermore, $C(\cdots)$, $T(\cdots)$, etc., stand for positive constants depending on the arguments in brackets.\\[-2mm]

 \noindent \textbf{Acknowledgments.} Yu Xiao and Can Zhang were partially supported by NSFC 12422118; Shengquan Xiang was partially supported by NSFC 12571474.

\section{Control design}\label{control}
 We temporarily impose the following \textit{compatibility condition}, denoted by \textbf{(H)}, on the initial datum:

 \noindent\textbf{(H).}\;\;\;The initial datum $(y_0,a_0)\in\mathcal{H}$ satisfies
\begin{equation}\label{cop-con}
    y_0(\cdot)-a_0\cos(\frac{\pi}{2}\;\cdot\;)\in H^1_0(0,1).
\end{equation}    
\noindent This condition enhances the regularity of $(w,a)$ (see Lemma \ref{Well-posedness-th}), thereby enabling the modal decomposition argument.\\[-2mm]

We introduce the operator $A$ by
\begin{equation}\label{Dirichlet-Laplace}
     Af= -f_{xx}, \quad \forall f\in D(A)=H^2(0,1)\cap H^1_0(0,1).
\end{equation}
Its eigenpairs are given by $\lambda_n=(n\pi)^2$ and $\varphi_n(x)=\sqrt{2}\sin(n\pi x)$. Here, $\{\varphi_n\}_{n\ge 1}$ forms an orthonormal basis of $L^2(0,1)$. Let $\psi(x)=\cos\frac{\pi}{2}x$, which satisfies 
\begin{equation*}
   \left\{\begin{array}{ll}
        \psi_{xx}(x)+ \frac{\pi^2}{4}\psi(x)=0\quad \text{for }x\in (0,1),\\
        \psi(0)=1\;\;\text{and}\;\;\psi(1)=0.
    \end{array}\right.
\end{equation*}

Define the operator $\mathcal{D}:\mathbb{R}\to L^2(0,1)$ by $\mathcal{D}a=\psi a$, and set
$$w(t)=y(t)-\mathcal{D} a(t),\;t\ge s.$$ For each $n\ge 1$, let $w_n(t)=\langle w(t),\varphi_n\rangle$. Let $N\in \mathbb{N}$ be an integer to be chosen later, and denote the \textit{low-frequency mode vector} by $ Y(t)=
(a(t),w_1(t) ,\cdots ,w_N(t))^\top\in \mathbb{R}^{N+1}$. 

Consider the following \textit{control design}:
\begin{equation}
    \label{con-heat}
       u(t)=KY(t),\;\;\forall t>s,
\end{equation}
where $K=[k_0,k_1,\cdots,k_N]\in \mathbb{R}^{1\times {(N+1)}}$ represents the controller gains. Equivalently, $u$ is given by the linear feedback law $\mathcal{K}_{\lambda}:\mathcal{H}\to \mathbb{R}$:
\begin{equation}\label{lin-feed}
    \mathcal{K}_{\lambda}(y,a)=k_0a+\sum_{j=1}^Nk_j\langle y-\mathcal{D}a,\varphi_j\rangle,\;\;\forall  (y,a)\in\mathcal{H}.
\end{equation}

The controlled system reads as
\begin{equation} \label{controlled-eqs}
    \left\{\begin{array}{ll}
 w_t(t)+A w (t)=w(t)-\mathcal{D}(KY(t))-\left(\mathcal{D}a(t)+w(t)\right)^3&\text{ for } t\in(s,+\infty),\\[2mm]
a_t(t)=(1- \frac{\pi^2}{4})a(t)+KY(t)&\text{ for } t\in  (s,+\infty),\\[2mm]
w(s )=w_0:=y_0-\mathcal{D}a_0\quad \text{and} \quad a(s)=a_0.
\end{array}\right.
\end{equation}
We consider the space $\mathcal{H}^1=H^1_0(0,1) \times\mathbb{R}$ with the norm $\|(w,a)\|_{\mathcal{H}^1}=\|w\|_{H^1_0(0,1)}+|a|$. Then the condition \textbf{(H)} gives $(w_0,a_0)\in\mathcal{H}^1$. The following lemma follows
from \cite[Lemma 5]{XZ}.
\begin{lemma}
\label{Well-posedness-th}
For each $(y_0,a_0)\in\mathcal{H}$ satisfying \textbf{(H)}, there exists $T_{\text{H}}=T_{\text{H}}\left(\|y_0\|_{H^1(0,1)},|a_0|\right)>0$ such that the system \eqref{controlled-eqs} admits a unique solution $(w,a)\in C([s,s+T_{\text{H}});\mathcal{H}^1)\cap C^1((s,s+T_{\text{H}});\mathcal{H})$ satisfying $w(t)\in H^2(0,1)\cap H^1_0(0,1)$ for $t\in(s,s+T_{\text{H}}).$ Moreover, if $T_{\text{H}}<+\infty$, then $\lim\limits_{t\to (s+T_{\text{H}})^-}\|(w(t),a(t))\|_{\mathcal{H}^1}=+\infty$.
\end{lemma}
Next, using the regularity of $(w,a)$, we perform the following modal decomposition: for every $n\ge 1$, 
 \begin{equation} \label{mode-eq-heat} 
 (w_n)_t(t)+\lambda_n w_n(t)=w_n(t)+b_n KY(t)+f_n(t),\;\;\forall t\in (s,s+T_{\text{H}}),
 \end{equation}
\begin{equation}\label{mode-eq-new-heat}
 Y_t(t)=(A_0+B_0K)Y(t)+R(t),\;\;\forall t\in (s,s+T_{\text{H}}),
\end{equation}
where  $$b_n=-\langle \psi,\varphi_n\rangle=\frac{-4\sqrt{2}n}{(4n^2-1)\pi},\;\; f_n(t)=\langle -\left(\mathcal{D}a(t)+w(t)\right)^3,\varphi_n\rangle, \quad \forall n\in \mathbb{N}^+,$$
and $$A_0=\text{diag}(1-\frac{\pi^2}{4},1-\lambda_1,\cdots,1-\lambda_N),\;B_0=(1,b_1,\cdots,b_N)^\top,\;R(t)=(0,f_1(t),\cdots,f_N(t))^\top.$$
By \cite[Lemma 4]{XZ}, the pair $(A_0,B_0)$ satisfies the \textit{Kalman rank condition}\footnote{For the definition of the Kalman rank condition, we refer to \cite[Chapter 1]{Trelat}.}. \\[-3mm]

We propose the following \textit{control strategy}.

\noindent \textbf{Selection of the number of modes $N$}: Let $\lambda>2\pi^2$ be a desired decay rate. We set 
$N= \lfloor \frac{\sqrt{2\lambda+1}}{\pi}\rfloor>1$. Since $\lambda_n=(n\pi)^2$, it holds that $\lambda_N>1$ and
\begin{equation}\label{N-sel}\left\{\begin{aligned}
    &\lambda_n\le 1+ 2\lambda,\;\;\forall n\le N,\\
    &\lambda_n>1+2\lambda,\;\;\forall n>N.
\end{aligned}\right.
\end{equation}

\noindent \textbf{Selection of the controller gains $K$}: Since $ (A_0,B_0)$ satisfies the Kalman rank condition,  the pole-placement theorem (cf. \cite[Theorem 17]{Trelat}) yields a matrix $K:=[k_0,k_1,\cdots,k_N]\in \mathbb{R}^{1\times (N+1)}$ such that the characteristic polynomial $\mathcal{X} _{A_0+B_0K}(\hat{\lambda})$ of $A_0+B_0K$ is $\mathcal{X} _{A_0+B_0K}(\hat{\lambda})=\prod_{n=0}^N(\hat{\lambda}-\hat{\lambda}_n)$, where
\begin{equation} \label{hat-lambda}
\hat{\lambda}_n=-3\lambda-\frac{\lfloor \lambda\rfloor}{N}n,\;\;\forall n\;=0,1,\cdots,N.\end{equation}
By applying \textit{Ackermann's formula} (see \cite[Section 3.2]{TK}), we obtain that
\begin{equation}\label{ac-f}
	K=-e_{N+1}^\top \mathcal{C}^{-1}(A_0,B_0)\mathcal{X}_{A_0+B_0K}(A_0),
\end{equation}
where $e_{N+1}^\top=(0,\cdots,0,1)\in\mathbb{R}^{{N+1}}$, $\mathcal{C}(A_0,B_0)=[B_0,A_0B_0,\cdots,A_0^NB_0]$, and $\mathcal{X} _{A_0+B_0K}(A_0)=\prod\limits_{n=0}^N(A_0-\hat{\lambda}_nI).$

Note that $K\neq 0$ since $\lambda>2\pi^2>\lambda_1$. With $N$ chosen via \eqref{N-sel} and $K$ via \eqref{ac-f}, we construct the finite-dimensional low-frequency feedback control \eqref{con-heat}, namely the feedback law \eqref{lin-feed}.

\section{Quantitative low-frequency estimates}\label{estimate}
This section is divided into three parts. We first recall the rapid stabilization argument, which reduces the quantitative problem to estimating $|K|_2$ and $\|e^{(A_0+B_0K)t}\|_2$. We then estimate $|K|_2$ by using Ackermann's formula \eqref{ac-f}, together with the inverse Vandermonde matrix estimate \eqref{V-est}. Finally, to estimate $\|e^{(A_0+B_0K)t}\|_2$, we diagonalize $A_0+B_0K$ and choose a suitable eigenvector matrix as in \eqref{S-sec}. This allows us to derive the desired quantitative bound by using the inverse Cauchy matrix decomposition \eqref{inverse-dec}.

\subsection{Strategy of the quantitative rapid stabilization}
We first recall the frequency Lyapunov method developed in \cite{X,XZ}, which incorporates a refined decomposition of low- and high-frequency components into the construction of the Lyapunov function and balances their respective contributions via suitable low-frequency estimates. 

Let us define
\begin{equation}\label{F-DEF}F:=
	A_0+B_0 K.\end{equation} 
With $N$ and $K$ chosen as above, $F+\lambda  I$ is a Hurwitz matrix. Then, there exists a \textit{positive definite} matrix $P$ satisfying the Lyapunov equation
\begin{equation}\label{ly-eq}
	F^\top P+PF+2\lambda  P=-I.
\end{equation}
Specifically, we adopt the following Lyapunov function
\begin{equation}\label{Lyapunov}
V(t)= \gamma Y(t)^{\top} P Y(t)+\sum_{n\ge N+1} \langle w(t),\varphi_n\rangle^2,\;\;t\ge s,
\end{equation}
where $ \gamma\ge \lambda_N>1$ is a parameter to be determined later. 

Note that the Lyapunov function $V$ is composed of two terms: the term $\gamma Y^\top P Y$ accounts for the low-frequency energy, while the remaining term corresponds to the high-frequency energy. Moreover, it immediately yields the equivalence relation
\begin{equation} \label{Psim-L2}
\min\{\sigma_{\min}(P),1\}\left(a^2(t)+\|w(t)\|^2_{L^2(0,1)}\right)\le V(t)\le  \max\{\gamma\sigma_{\max}(P),1\}\left(a^2(t)+\|w(t)\|^2_{L^2(0,1)}\right).
 \end{equation}
Using the fact that $P$ satisfies the Lyapunov equation \eqref{ly-eq}, \cite{XZ} proved that there exists an explicit constant $\gamma$ such that, for sufficiently small initial datum, $V_t(t)+2\lambda V(t)\le 0$ for $t> s$. Combined with \eqref{Psim-L2}, this directly implies
$$a^2(t)+\|w(t)\|^2_{L^2(0,1)}\le \frac{\max\{\gamma\sigma_{\max}(P),1\}}{\min\{\sigma_{\min}(P),1\}}e^{-2\lambda (t-s)}\left(a^2_0+\|w_0\|^2_{L^2(0,1)}\right),\;\;\forall t\ge s.$$
Therefore, estimating $\sigma_{\max}(P)$ and $\sigma_{\min}(P)$ is necessary for quantifying the stabilization cost.

Since $P$ is the solution to \eqref{ly-eq}, we have the following explicit expression:
$$P=\int_0^{\infty}e^{(F+\lambda I)^{\top}\tau}e^{(F+\lambda I)\tau}d\tau.$$ We deduce that
\begin{equation}\label{P-max}
\begin{aligned}
		\sigma_{\max}(P)=\|P\|_2\le \int_0^{+\infty} e^{2\lambda \tau}\|e^{F\tau}\|_2^2\;d\tau.
\end{aligned}
\end{equation} 

Furthermore, the following positive-semidefinite matrix inequality holds:
\begin{equation*}
	\left(2(F+\lambda I)^{\top}P+\frac{1}{2}I\right)^{\top}\left(2(F+\lambda I)^{\top}P+\frac{1}{2}I\right)\succeq 0.
\end{equation*}
Expanding this and using \eqref{ly-eq}, we derive $P(F+\lambda I)(F+\lambda I)^{\top}P-\frac{3}{16}I\succeq 0$. For all $x\in \mathbb R^{N+1}$ with $|x|_2=1$, we have
\begin{equation*}
	\frac{1}{16}< x^\top P(F+\lambda I)(F+\lambda I)^{\top}Px\le \|F+\lambda I\|_2^2|Px|_2^2.
\end{equation*}
Consequently,
\begin{equation}\label{P-min}
\begin{aligned}
\sigma_{\min}(P)&> \frac{1}{4\|F+\lambda I\|_2}=\frac{1}{4\|A_0+B_0K+\lambda I\|_2}.
\end{aligned}
\end{equation}

Therefore, to achieve quantitative rapid stabilization, it remains to establish quantitative estimates of $|K|_2$ and $\|e^{Ft}\|_2$.

\subsection{Quantitative estimate of $|K|_2$}
By Ackermann's formula \eqref{ac-f}, we obtain
\begin{equation*}
|K|_2\le \|\mathcal{C}^{-1}(A_0,B_0)\|_2\prod\limits_{n=0}^N\|A_0-\hat{\lambda}_nI\|_2.
\end{equation*}

Let $\mu_0=1-\frac{\pi^2}{4}$ and $ \mu_n=1-\lambda_n,\forall n\in\mathbb{N}^+.$ It follows that $$\mathcal{C}(A_0,B_0)=\text{diag}(1,b_1,\cdots,b_N)\text{VdM}(\mu_0,\mu_1,\cdots,\mu_N),$$ where $\text{VdM}(\mu_0,\mu_1,\cdots,\mu_N)$ stands for the Vandermonde matrix. 
Consequently, we derive that
\begin{equation*}
|K|_2\le \|\text{diag}(1,b^{-1}_1,\cdots,b^{-1}_N)\|_2\|\text{VdM}^{-1}(\mu_0,\mu_1,\cdots,\mu_N)\|_2\prod\limits_{n=0}^N\|A_0-\hat{\lambda}_nI\|_2.
\end{equation*}

Since $b_n=\frac{-4\sqrt{2}n}{(4n^2-1)\pi}$, a direct computation gives $\|\text{diag}(1,b^{-1}_1,\cdots,b^{-1}_N)\|_2=\mathcal{O}(N)$. Moreover, given that $$-4\lambda\le \hat{\lambda}_n\le -3\lambda \quad\text{and} \quad-2\lambda\le\mu_n\le 1-\frac{\pi^2}{4}<0,\quad  \forall n=0,1,\cdots,N,$$ we obtain
$\prod\limits_{n=0}^N\|A_0-\hat{\lambda}_nI\|_2\le (4\lambda)^{N+1}$.
It remains to estimate $\|\text{VdM}^{-1}(\mu_0,\mu_1,\cdots,\mu_N)\|_2$. By the inverse Vandermonde estimate in \cite{WG}, we have
\begin{equation}\label{V-est}
\|\text{VdM}^{-1}(\mu_0,\mu_1,\cdots,\mu_N)\|_2\le (N+1)\|\text{VdM}^{-1}(\mu_0,\mu_1,\cdots,\mu_N)\|_{\infty}\le (N+1)\max_{0\le j\le N}\prod_{k\neq j}\frac{1+|\mu_k|}{|\mu_k-\mu_j|}.
\end{equation}

For $j=0$, it holds that
\begin{equation*}
 \prod_{k\neq 0}\frac{1+|\mu_k|}{|\mu_k-\mu_0|}=\prod_{k=1}^N\frac{k^2}{|k^2-\frac{1}{4}|}\le (\frac{ 4}{3})^{N}=e^{\mathcal{O}(N)}.
\end{equation*}
For $j\neq 0$, we have
\begin{equation*}
\prod_{k\neq j}\frac{1+|\mu_k|}{|\mu_k-\mu_j|}\le \frac{1}{3} \prod_{k=1,k\neq j}^N\frac{k^2}{|k^2-j^2|}< \frac{(N!)^2}{(N+j)!(N-j)!}<1.
\end{equation*}

Therefore, we obtain
\begin{equation*}
\|\text{VdM}^{-1}(\mu_0,\mu_1,\cdots,\mu_N)\|_2\le (N+1)\max_{0\le j\le N}\prod_{k\neq j}\frac{1+|\mu_k|}{|\mu_k-\mu_j|}= e^{\mathcal{O}(N)}.
\end{equation*}
Finally, we derive 
\begin{equation}\label{K-norm}
|K|_2\le (4\lambda)^{N+1}e^{\mathcal{O}(N)}=e^{\mathcal{O}(\sqrt{\lambda}\ln \lambda)}.
\end{equation}

\subsection{Quantitative estimate of $\|e^{Ft}\|_2$}

Since the eigenvalues of $F$ are distinct, $F$ is diagonalizable. Thus, there exists an \textit{invertible} matrix $S=[v_0,v_1,\cdots,v_N]$ such that
$F=S\Lambda S^{-1}$, where $v_j$ is the eigenvector of $F$ corresponding to $\hat{\lambda}_j$, and $\Lambda=\text{diag}(\hat{\lambda}_0,\hat{\lambda}_1,\cdots,\hat{\lambda}_N)$. It follows that $\|e^{Ft}\|_2\le \|S^{-1}\|_2\|S\|_2 e^{-3\lambda t}.$

From $(A_0+B_0K)v_j=\hat{\lambda}_jv_j$, we derive 
$v_j=-(A_0-\hat{\lambda}_jI)^{-1}B_0Kv_j$ for all $j=0,1,\cdots,N$. Note that $Kv_j\in \mathbb{R}$ for each $j$. This motivates the choice
$$v_j=(A_0-\hat{\lambda}_jI)^{-1}B_0\neq 0,\quad\forall j=0,1,\cdots,N.$$ The next lemma shows that, for each $j=0,1,\cdots,N$, $v_j$ is indeed an eigenvector of $F$ corresponding to $\hat{\lambda}_j$.

\begin{lemma}\label{eigvecof-F1}
For each $j=0,1,\cdots,N$, let $v_j=(A_0-\hat{\lambda}_jI)^{-1}B_0$. It holds that
$Fv_j=\hat{\lambda}_jv_j.$
\end{lemma}
\begin{proof}
Since $Fv_j=(A_0+B_0K)(A_0-\hat{\lambda}_jI)^{-1}B_0=B_0+B_0K(A_0-\hat{\lambda}_jI)^{-1}B_0+\hat{\lambda}_jv_j$, we only need to prove
$K(A_0-\hat{\lambda}_jI)^{-1}B_0=-1$. Indeed, we have
\begin{equation*}
\begin{aligned}
 0&=\text{det}\left(F-\hat{\lambda}_jI\right)=(1+K(A_0-\hat{\lambda}_jI)^{-1}B_0)\text{det}(A_0-\hat{\lambda}_jI).
\end{aligned}
\end{equation*}
Recalling that $\text{det}(A_0-\hat{\lambda}_jI)\neq 0$, we derive $1+K(A_0-\hat{\lambda}_jI)^{-1}B_0=0$. This completes the proof.
\end{proof}
Therefore, we can choose 
\begin{equation}\label{S-sec}
S=[(A_0-\hat{\lambda}_0I)^{-1}B_0,(A_0-\hat{\lambda}_1I)^{-1}B_0,\cdots,(A_0-\hat{\lambda}_NI)^{-1}B_0].\end{equation} Observe that $S$ admits the decomposition $$S=\text{diag}(1,b_1,\cdots,b_N)\mathcal{S},\;\;\text{ where }
\mathcal{S}=(\mathcal{S}_{kj})_{0\le k,j\le N},\;\text{and}\;\mathcal{S}_{kj}=\frac{1}{\mu_k-\hat{\lambda}_j}.$$
Hence, $\mathcal S$ is a Cauchy matrix. Using the fact that $|\mathcal{S}_{kj}|\le \frac{1}{\lambda}<1$, we derive
\begin{equation} \label{T-2}
\|S\|_2\le \sqrt{1+\|\psi\|_{L^2(0,1)}^2}\|\mathcal{S}\|_{2}< \sqrt{1+\|\psi\|_{L^2(0,1)}^2}(N+1).
\end{equation}

To estimate $\|S^{-1}\|_2$, it suffices to estimate $\|\mathcal{S}^{-1}\|_2$. Define two auxiliary quantities
$$A_j = \prod_{\substack{m=0 \\ m \neq j}}^{N} (\mu_j - \mu_m), \quad B_k = \prod_{\substack{m=0 \\ m \neq k}}^{N} (-\hat{\lambda}_k + \hat{\lambda}_m).$$
Then $\mathcal{S}^{-1}$ has the following explicit form (see, e.g., \cite{Ss})
\begin{equation}\label{inverse-dec}
(\mathcal{S}^{-1})_{kj} = \frac{\prod_{m=0}^{N} (\mu_j - \hat{\lambda}_m)(\mu_m - \hat{\lambda}_k)}{A_j B_k(\mu_j - \hat{\lambda}_k)} .
\end{equation}
Note that $|A_j|\ge 1$. Since $|\hat{\lambda}_m-\hat{\lambda}_k|\ge \frac{\lfloor\lambda\rfloor}{N}|m-k|$, we have $|B_k|^{-1}\le \left(\frac{N}{\lfloor\lambda\rfloor}\right)^N\frac{1}{k!(N-k)!}\le N^N$. Consequently, we deduce that
\begin{equation}\label{T-1-2}
\|\mathcal{S}^{-1}\|_2\le (N+1)N^N(4\lambda)^{2N+2}=e^{\mathcal{O}(\sqrt{\lambda}\ln \lambda)}.
\end{equation}
Combining this with \eqref{T-2} and $\|\text{diag}(1,b^{-1}_1,\cdots,b^{-1}_N)\|_2=\mathcal{O}(N)$ yields
\begin{equation} \label{F-NORM}
\|e^{Ft}\|_2\le e^{\mathcal{O}(\sqrt{\lambda}\ln \lambda)}e^{-3\lambda t}.
\end{equation}

From \eqref{P-max}, \eqref{P-min}, \eqref{K-norm}, and \eqref{F-NORM}, we obtain the following result.
   
\begin{proposition}[Quantitative low-frequency estimates]
Let $\lambda>2\pi^2$. For the previously defined $N$, $A_0$, and $B_0$, let $K$ be defined as in \eqref{ac-f}. Then, for the solution $P$ of the Lyapunov equation \eqref{ly-eq}, there exist $\tilde{D}_1,\tilde{D}_2>1$ such that
\begin{equation}\label{P-minmax}
   \sigma_{\max}(P)\le \sigma_1:=e^{\tilde{D}_1\sqrt{\lambda}\ln \lambda},\quad\quad \sigma_{\min}(P)\ge \sigma_2:=e^{-\tilde{D}_2\sqrt{\lambda}\ln \lambda}.
\end{equation}
\end{proposition}
   
\begin{remark}
   The pole-placement theorem ensures that there exists a matrix $K $ such that 
  $$\|e^{(A_0+B_0K)t}\|_2 \le M_\lambda e^{-\lambda t},\quad \forall t\ge0.$$
   The investigation of the quantitative relationship between $M_\lambda$ and  $\lambda$ is referred to as the \textit{overshoot estimation} in pole placement. Related results can be found in \cite{J}.
\end{remark}
\section{Proof of quantitative rapid stabilization; Theorem \ref{qu-ra-sta}}\label{Stabilizatipn}

Recall the frequency Lyapunov function defined in \eqref{Lyapunov}; the proof is divided into three steps.\\[-1mm]

\noindent\textbf{Step 1. We prove the rapid stabilization of $V(t)$ over $(s,s+T_{\text{H}})$ under the compatibility condition.}\\[-2mm]

By the well-posedness result in Lemma \ref{Well-posedness-th}, the solution is sufficiently regular on $(s,s+T_{\text{H}})$ for the following computation:
\begin{equation}\label{V-along-traj-L2}
\begin{aligned}
V_t(t)+2\lambda V(t)&=   - \gamma Y^{\top}(t)Y(t)+2\gamma Y^{\top}(t)P R(t)+2 \sum_{n\ge N+1} w_{n}(t) \{(1+\lambda-\lambda_n )w_n(t)\\
&\quad+b_n KY(t)+f_n(t)\},\;\;\forall t\in (s,s+T_{\text{H}}).
\end{aligned}
\end{equation}
Using Young's inequality, for all $\alpha>0$, we have
$$\begin{aligned}
2 \sum_{n\ge N+1}w_{n}(t)f_n(t)&\le  \sum_{n\ge N+1}\{\alpha  {f_n^2(t)}+\frac{1}{\alpha  }w_{n}^2(t)\}\\[2mm]
&=\frac{1}{\alpha  } \sum_{n\ge N+1}w_{n}^2(t)+\alpha \|(\mathcal{D}a(t)+w(t))^3\|_{L^2(0,1)}^2-{\alpha  }R^{\top}(t)R(t),\;\;\forall t\in (s,s+T_{\text{H}}).
\end{aligned}$$

For every  $\varphi\in H^1(0,1)$ with $\varphi(1)=0$, we have
$$\varphi^2(x)= -2\int_x^1\varphi(z)\varphi'(z)dz.$$
By direct computations, it holds that
\begin{equation}\label{AG610}\|\varphi\|_{L^{\infty}(0,1)}\le \sqrt{2}\|\varphi\|_{L^{2}(0,1)}^{1/2}\|\varphi_x\|_{L^{2}(0,1)}^{1/2}.\end{equation}
We further have
$$ \|(\mathcal{D}a(t)+w(t))^3\|_{L^2(0,1)}^2\le 4\|(\mathcal{D}a(t)+w(t))_x\|_{L^2(0,1)}^2\|\mathcal{D}a(t)+w(t)\|^4_{L^2(0,1)}.$$

Hence, for every $t\in (s,s+T_{\text{H}})$, we obtain
\begin{equation*}
\begin{aligned}
{V}_t(t)+2\lambda V(t)&\le Y^{\top}(t)\Phi Y(t)+2\gamma Y^{\top}(t)P R(t)-{\alpha}R^{\top}(t)R(t)+\sum_{n\ge N+1}\Psi_n w_{n}^2(t)\\
&\quad +32\alpha \left(a^2(t)+\|w_x(t)\|^2_{L^2(0,1)}\right)\left(a^2(t)+\|w(t)\|^2_{L^2(0,1)}\right)^2,
\end{aligned}
\end{equation*}
where 
\begin{equation*}
\begin{aligned}
&\Phi=(-\gamma+ \beta \sum_{n\ge N+1}b_n^2|K|_2^2)I:=(-\gamma + \beta S_{N})I,\\[2mm]
&\Psi_n=2+2\lambda-2\lambda_n+\frac{1}{\beta}+\frac{1}{\alpha},\;\;\forall n\ge N+1.
\end{aligned}
\end{equation*}

Equivalently,
$$\begin{aligned}
{V}_t(t)+2\lambda V(t)&\le \begin{pmatrix}
Y(t)\\R(t)\end{pmatrix}^T\Theta\begin{pmatrix}Y(t)\\R(t)\end{pmatrix}+\sum_{n\ge N+1}\Psi_n w_{n}^2(t)\\
&\quad +32\alpha \left(a^2(t)+\|w_x(t)\|^2_{L^2(0,1)}\right)\left(a^2(t)+\|w(t)\|^2_{L^2(0,1)}\right)^2,  \end{aligned}$$
where $\Theta:=\begin{pmatrix}
\Phi &  \gamma P\\
\gamma P  & -\alpha I
\end{pmatrix}$.

Note that $S_N>0$. Set $\beta=\frac{\gamma}{4S_{N}}$. Then 
$\Phi=(-\gamma+\beta S_N)I=-\frac{3\gamma}{4}I.$ We choose $\alpha$ such that $\alpha\ge 4\gamma \sigma^2_1+\frac{\gamma}{2}.$ It follows that 
\begin{equation}\label{Phi-e1}
-\Phi-\frac{\gamma}{2}I-\frac{\gamma^2 }{\alpha-\frac{\gamma}{2}}P^2\succeq 0.
\end{equation}
By the Schur complement, \eqref{Phi-e1} implies that $-\Theta- \frac{\gamma}{2}I\succeq0$, and thus $-\Theta- \frac{\lambda_N}{2}I\succeq0$. For $\Psi_{n}$, we have
$$\Psi_{n}\le 2+2\lambda-2\lambda_{n}+\frac{4S_{N}}{\gamma}+\frac{1}{4\gamma \sigma^2_1+\frac{\gamma}{2}}\le 2+2\lambda-2\lambda_{n}+\frac{4S_{N}}{\gamma}+\frac{2}{\gamma},\quad \forall n\ge N+1.$$

Letting $\gamma>8S_N+4$, so that $\frac{4S_N}{\gamma}+\frac{2}{\gamma}<\frac12$, and using $\lambda_n>1+2\lambda$ for $n\ge N+1$, we deduce that 
$$\Psi_n+\frac{1}{2}\lambda_n\le 2+2\lambda-\frac{3}{2}\lambda_n+\frac{1}{2}\le 1-\lambda<0,\quad\forall n\ge N+1.$$
Consequently,
\begin{equation}\label{V-nes}\begin{aligned}
\begin{pmatrix}Y(t)\\R(t)\end{pmatrix}^T\Theta\begin{pmatrix}Y(t)\\R(t)
\end{pmatrix}+\sum_{n\ge N+1}\Psi_n w_{n}^2(t)&\le -\frac{\lambda_N}{2}Y^\top (t)Y(t)-\frac{1}{2}\sum_{n\ge N+1}\lambda_nw_n^2(t)\\
&\le -\frac{1}{2} \left(a^2(t)+\|w_x(t)\|_{L^2(0,1)}^2\right).\end{aligned}
\end{equation} 

Note that $S_N=\sum\limits_{n\ge N+1}b_n^2|K|_2^2\le \|\psi\|_{L^2(0,1)}^2|K|_2^2$. It follows from \eqref{K-norm} that there exists $\tilde D_3\ge 1$ such that the choice $\gamma=e^{\tilde D_3\sqrt{\lambda}\ln\lambda}$ ensures
the estimate \eqref{V-nes} and the inequality $\gamma>\max\{8S_N+4,\lambda_N\}$. Moreover, by \eqref{P-minmax}, after increasing the constant if necessary, there exists $\tilde D_4\ge 1$ such that we may take $\alpha=e^{\tilde D_4\sqrt{\lambda}\ln\lambda}$ satisfying $\alpha\ge 4\gamma\sigma_1^2+\frac{\gamma}{2}$.

Therefore, we derive
\begin{equation*}
\begin{aligned}
{V}_t(t)+2\lambda V(t)&\le \left(-\frac{1}{2}+32\alpha \left(a^2(t)+\|w(t)\|^2_{L^2(0,1)}\right)^2\right)\left(a^2(t)+\|w_x(t)\|^2_{L^2(0,1)}\right),\;\;\forall t\in (s,s+T_{\text{H}}).
\end{aligned}
\end{equation*}
Suppose that
\begin{equation} \label{pri}
\|(w(t),a(t))\|_{\mathcal{H}}\le \kappa:=\left(\frac{1}{64\alpha}\right)^{\frac{1}{4}}=2^{-\frac 3 2}e^{-\frac{\tilde D_4}{4}\sqrt{\lambda}\ln\lambda},\quad\forall t\in(s,s+T_{\text{H}}).
\end{equation}
It follows that
$$\begin{aligned}\|(w(t),a(t))\|_{\mathcal{H}}&\le \left( \frac{2\gamma\sigma_1}{\sigma_2}\right)^{\frac{1}{2}}e^{-\lambda (t-s)}\|(w_0,a_0)\|_{\mathcal{H}}\\
&\le e^{(\tilde{D}_3+\tilde{D}_1+ \tilde{D}_2/2)\sqrt{\lambda}\ln\lambda}e^{-\lambda (t-s)}\|(w_0,a_0)\|_{\mathcal{H}},\;\;\forall t\in(s,s+T_{\text{H}}).\end{aligned}$$

Let $$M_0=e^{(\tilde{D}_3+\tilde{D}_1+\tilde{D}_2 /2)\sqrt{\lambda}\ln\lambda}\quad\text{and}\quad \rho_0={{{\kappa}}/{M_0}}.$$ 
We claim that the a priori estimate \eqref{pri} holds whenever $(y_0,a_0)\in H^1(0,1)\times\mathbb{R}$ satisfies \textbf{(H)} and $\|(y_0-\mathcal{D}a_0,a_0)\|_{\mathcal{H}}\le \rho_0$. 

Indeed, we suppose that there exists $\tilde{t}\in [s,s+T_{\text{H}})$ such that $\|(w(\tilde t),a(\tilde t))\|_{\mathcal{H}}\ge \kappa$. When $t=s$, we have
$\|(w_0,a_0)\|_{\mathcal{H}}=\|(y_0-\mathcal{D}a_0,a_0)\|_{\mathcal{H}}\le \rho_0<\kappa$. Thus, we define $t^*\in(s, \tilde{t}]$ as the shortest time satisfying
$\|(w(t^*),a(t^*))\|_{\mathcal{H}}\ge\kappa$. Using the fact that $(w,a)\in C([s,s+T_{\text{H}});\mathcal{H}^1)$, we have $\|(w(t^*),a(t^*))\|_{\mathcal{H}}=\kappa$, and $\|(w(t),a(t))\|_{\mathcal{H}}<\kappa$ for any $t\in [s,t^*)$. Hence, $V(t)$ is exponentially stable over $[s,{t}^*)$. This leads to
\begin{equation*}
\begin{aligned}
\|(w(t),a(t))\|_{\mathcal{H}}&\le M_0e^{-\lambda (t-s)}\|(y_0-\mathcal{D}a_0,a_0)\|_{\mathcal{H}},\;\;\forall t\in[s,{t}^*).
\end{aligned}
\end{equation*}
Hence, using the continuity again, we obtain
\begin{equation*}
\|(w(t^*),a(t^*))\|_{\mathcal{H}}\le M_0e^{-\lambda (t^*-s)}\|(y_0-\mathcal{D}a_0,a_0)\|_{\mathcal{H}} \le e^{-\lambda (t^*-s)}\kappa<\kappa,
\end{equation*}
which contradicts the definition of $t^*$. 

Therefore, we have
\begin{equation} \label{y-est-l2}
\|(w(t),a(t))\|_{\mathcal{H}}\le M_0e^{-\lambda(t-s)}\|(y_0-\mathcal{D}a_0,a_0)\|_{\mathcal{H}},\quad\forall t\in(s,s+T_{\text{H}}).
\end{equation} Furthermore, 
\begin{equation}\label{y-est-l222}
|u(t)|\le |K|_2\|(w(t),a(t))\|_{\mathcal{H}}\le |K|_2M_0e^{-\lambda(t-s)}\|(y_0-\mathcal{D}a_0,a_0)\|_{\mathcal{H}},\quad\forall t\in(s,s+T_{\text{H}}).
\end{equation} 

\noindent\textbf{Step 2. We extend the exponentially stable solution to $[s,+\infty)$.}\\[-2mm]

 A standard multiplier argument yields
$$ \begin{aligned}
\dfrac{d}{dt}\left(\|y(t)\|^2_{H^1(0,1)}\right)&\le -2y_x(t,0)a_t(t)-\|y_{xx}(t)\|^2_{L^2(0,1)}+\|y(t)\|_{L^{\infty}}^4\|y(t)\|_{L^2(0,1)}^2+2\|y(t)\|_{L^2(0,1)}^2.
\end{aligned}$$
From \eqref{AG610} and Sobolev's embedding theorem, there exists $\tilde{c}>0$ such that
\begin{equation}\label{global-ti2}
\begin{aligned}
\dfrac{d}{dt}\left(\|y(t)\|^2_{H^1(0,1)}\right)&\le 2\tilde{c}\sqrt{\|y_{xx}(t)\|_{L^2(0,1)}^2+\|y_{x}(t)\|_{L^2(0,1)}^2}|a_t(t)|-\|y_{xx}(t)\|^2_{L^2(0,1)}\\
&\quad+4\|y_x(t)\|_{L^2(0,1)}^2\|y(t)\|_{L^2(0,1)}^4+2\|y(t)\|_{L^2(0,1)}^2.
\end{aligned}
\end{equation}
Moreover, by \eqref{y-est-l2}, \eqref{y-est-l222}, and $a_t=(1-\pi^2/4)a+u$, there exists a constant $C_\lambda>0$ independent of $T_{\text{H}}$ such that
\begin{equation}\label{global-ti68}\|y(t)\|_{L^2}+|a(t)|+|a_t(t)|
\le C_\lambda \|(y_0-\mathcal Da_0,a_0)\|_{\mathcal H},\quad\forall t\in (s,s+T_{\text{H}}).\end{equation}

Using Young's inequality in \eqref{global-ti2}, we derive 
$$\dfrac{d}{dt}\left(\|y(t)\|^2_{H^1(0,1)}\right)\le C\bigl(1+\|y(t)\|_{L^2}^4\bigr)\|y(t)\|^2_{H^1(0,1)}
+C\bigl(|a_t(t)|^2+\|y(t)\|_{L^2}^2\bigr),$$
for some $C>0$. From \eqref{global-ti68} and Gr\"onwall's inequality, we obtain that $\|y(t)\|_{H^1(0,1)}$ and $\|(w(t),a(t))\|_{\mathcal{H}^1}$ are uniformly bounded for all $t\in(s,s+T_{\text{H}})$. In view of the blow-up alternative in {Lemma \ref{Well-posedness-th}}, the solution can be extended to $[s,+\infty)$.\\[-2mm]

\noindent \textbf{Step 3. We prove the quantitative rapid stabilization with $(y_0,a_0)\in \mathcal{H}$.}\\[-2mm]

Define $G_\lambda(w,a):=\mathcal K_\lambda (w+\mathcal Da,a)$. It follows directly that $G_\lambda$ is Lipschitz continuous. According to Lemma \ref{w2}, for every $(y_0,a_0)\in \mathcal{H}$, there exists ${T}_1>0$ such that the system \eqref{controlled-eqs} admits a unique solution
$(w,a)\in C([s,s+{T}_1];\mathcal{H})\cap L^2(s,s+{T}_1;\mathcal{H}^1)$
satisfying
$$\|(w,a)\|_{C([s,s+{T}_1];\mathcal{H})}\le 2\|(w_0,a_0)\|_{\mathcal{H}}=2\|(y_0-\mathcal{D}a_0,a_0)\|_{\mathcal{H}}\le 4\|(y_0,a_0)\|_{\mathcal{H}}.$$
Let us set 
\begin{equation}\label{rhoM-57}
\rho=\frac{\rho_0}{4}\quad \text{and}\quad M= 16M_0e^{\sqrt{\lambda}}.
\end{equation}
Then, if $(y_0,a_0)\in \mathcal{H}$ satisfies
$\| (y_0,a_0)\|_{\mathcal{H}}\le \rho$, we have \begin{equation}\label{ndw67}\|(w,a)\|_{C([s,s+{T}_1];\mathcal H)}=\|(y-\mathcal Da,a)\|_{C([s,s+{T}_1];\mathcal H)}\le4\|(y_0,a_0)\|_{\mathcal{H}}\le \rho_0.\end{equation}

On one hand, for $t\in [s,s+\min\{{T}_1,\frac{1}{\sqrt{\lambda}}\}]\subset[s,s+\frac{1}{\sqrt{\lambda}}]$, we have
\begin{equation}\label{M167}
\|(y(t),a(t))\|_{\mathcal{H}}\le 2\|a\|_{C([s,s+{T}_1])}+\|w\|_{C([s,s+{T}_1];L^2(0,1))} \le 8e^{\sqrt{\lambda } }e^{-\lambda  (t-s)}\|(y_0,a_0)\|_{\mathcal{H}}.
\end{equation}
On the other hand, since $w\in L^2(s,s+{T}_1;H^1_0(0,1))$, there exists $\tilde{T}\in (s,s+\min\{{T}_1,\frac{1}{\sqrt{\lambda}}\})$ such that $w(\tilde{T})\in H_0^1(0,1)$. Hence, $(y(\tilde{T}),a(\tilde{T}))$ satisfies the compatibility condition. It follows from \eqref{ndw67} that
\begin{equation}\label{M267}\begin{aligned}\|(y(t),a(t))\|_{\mathcal{H}}\le 2\|(w(t),a(t))\|_{\mathcal{H}}&\le 2M_0e^{-\lambda (t-\tilde{T})}\|(y(\tilde{T})-\mathcal{D}a(  \tilde{T}),a (  \tilde{T}))\|_{\mathcal{H}}\\
&\le 2M_0e^{\sqrt{\lambda } }e^{-\lambda (t-s)}\|(y(\tilde{T})-\mathcal{D}a(  \tilde{T}),a (  \tilde{T}))\|_{\mathcal{H}}\\
&\le 16M_0e^{\sqrt{\lambda } }e^{-\lambda (t-s)}\|(y_0,a_0)\|_{\mathcal{H}},\;\;\forall t\ge \tilde{T}.\end{aligned}\end{equation}

Therefore, \eqref{M167} and \eqref{M267} yield
\begin{equation*}
\|(y(t),a(t))\|_{\mathcal{H}}\le Me^{-\lambda (t-s)}\|(y_0,a_0)\|_{\mathcal{H}},\;\;\;\forall t\ge s.
\end{equation*}
Moreover, it holds that
$$|u(t)|\le |K|_2Me^{-\lambda (t-s)}\|(y_0,a_0)\|_{\mathcal{H}},\;\;\forall t\ge s.$$
The quantitative estimates in Step 1 imply that there exists $D\ge 1$ such that
$$(1+|K|_2)M\le e^{D{\sqrt{\lambda}\ln\lambda}} \quad \text{and}\quad\rho\ge e^{-D{\sqrt{\lambda}\ln\lambda}}.$$
We complete the proof. \hfill$\square$
\begin{remark}\label{quant-re}
Compared with the stabilization cost $e^{\mathcal{O}(\sqrt{\lambda})}$ presented in \cite{CN,X}, our stabilization cost behaves as $e^{\mathcal{O}(\sqrt{\lambda}\ln\lambda)}$. This is because our proof relies heavily on the estimate $\lambda^N=e^{\mathcal{O}(\sqrt{\lambda}\ln\lambda)}.$  
This factor arises from the terms $|\mu_k-\hat{\lambda}_j|$, which represent the distances between low-frequency eigenvalues and pole-placement eigenvalues.

Indeed, for every $p>1$ we have $p/(p+1)>1/2$, and hence there exists a constant $C_p>0$ such
that
$\sqrt\lambda\ln\lambda\le C_p\lambda^{\frac{p}{p+1}}$ for any $ \lambda>2\pi^2.$
Consequently, after enlarging the constant if necessary, the estimates in Theorem \ref{qu-ra-sta} imply that we can set
$$\rho_{\lambda}=e^{-D_p\lambda^{\frac{p}{p+1}}},\quad
M_{\lambda}=e^{D_p\lambda^{\frac{p}{p+1}}},\quad
\|\mathcal K_\lambda\|_{\mathcal L}\le e^{D_p\lambda^{\frac{p}{p+1}}}.$$
\end{remark}
\begin{remark}
We obtain the quantitative local rapid stabilization in $\mathcal{H}$. In fact, for many nonlinear parabolic equations, it is natural to work in a \textit{higher-regularity} space (e.g., $\mathcal{H}^1$ in \cite{XZ}). To this end, we introduce the Lyapunov function
\begin{equation}\label{Lyapunov-H1}
\mathcal{V}(t)= \gamma Y(t)^{\top}P Y(t)+\sum_{n\ge N+1} \lambda_n \langle w(t),\varphi_n\rangle^2,\;\;t\ge s.
\end{equation}
Then, combining the quantitative framework developed in this paper with the argument for rapid stabilization in \cite{XZ}, we further obtain the quantitative local rapid stabilization in $\mathcal{H}^1$.
\end{remark}
\begin{remark}
For linear parabolic equations, the global well-posedness of the closed-loop system follows from standard bounded perturbation theory. Thus, regardless of whether one uses the Lyapunov function \eqref{Lyapunov} or \eqref{Lyapunov-H1}, the resulting quantitative results are global.
\end{remark}

\begin{appendices}
\titleformat{\section}
{\normalfont\bfseries} 
{{\large Appendix \thesection.}}  
{1em}
{}
\section{\large{Well-posedness results and the maximum principle}}
Consider the linear inhomogeneous heat equation
\begin{equation}\label{in-h}
 \left\{\begin{array}{ll}
        y_t(t,x)-y_{xx}(t,x)=f(t,x)&\text{for }(t,x)\in(t_1,t_2)\times(0,1),\\
        y(t,0)=\gamma(t),\quad y(t,1)=\beta(t)&\text{for }t\in(t_1,t_2),\\
        y(t_1,x)=y_0(x)&\text{for }x\in(0,1).\\
    \end{array}\right.
\end{equation}
The following two well-known results hold (see \cite{CN} and \cite{CX}, respectively).
\begin{lemma}
    \label{well-po-for-in-h}
    If $y_0\in L^2(0,1),\;\beta=\gamma=0$, and $f\in L^1(t_1,t_2;L^2(0,1))$, then \eqref{in-h} admits a unique mild solution 
    $ y\in C([t_1,t_2];L^2(0,1))\cap L^2(t_1,t_2;H^1_0(0,1)) $ satisfying
    \begin{equation}
        \label{norm-est-in-h}
        \max\Bigl\{\|y\|_{C([t_1,t_2];L^2(0,1))},  \|y\|_{L^2(t_1,t_2;H^1_0(0,1))}\Bigr\}\le \|y_0\|_{L^2(0,1)}+\|f\|_{L^1(t_1,t_2;L^2(0,1))}.
    \end{equation}
\end{lemma}
\begin{lemma}[Maximum principle]\label{max-for-in-h}
Let $y_0\in L^2(0,1),\;\beta,\gamma\in L^2(t_1,t_2)$, and $f\in L^2(t_1,t_2;L^2 (0,1))$ be such that
$$y_0\ge 0,\;\beta\ge 0,\;\gamma\ge 0,\;f\ge 0.$$
Then, for every $t\in [t_1,t_2]$, we have $y(t,\cdot)\ge 0.$
\end{lemma}

The next two lemmas follow from the arguments in Coron-Xiang \cite[Lemmas 25, 34 and 39]{CX}. We only give a sketch of the proof for Lemma \ref{w2}.

\begin{lemma}\label{w2}
Let $L>0$, and let $G:\mathcal{H}\to\mathbb{R}$ be a feedback law satisfying $G(0)=0$ and
$$|G(w_1,a_1)-G(w_2,a_2)|\le L\|(w_1-w_2,a_1-a_2)\|_{\mathcal{H}},\;\forall (w_j,a_j)\in \mathcal{H},\;j=1,2.$$
Then, for every $r>0$, there exists $T_1=T_1\left(r,L\right)>0$ such that for every $0\le t_1<t_2\le t_1+T_1$, and for every $(w_0,a_0)\in \mathcal{H}$ satisfying $\|(w_0,a_0)\|_{\mathcal H}\le r$, the system
\begin{equation}\label{we1}
\left\{\begin{array}{ll}
    w_t(t)+Aw(t)=w(t)-\mathcal{D}(G(w(t),a(t)))-(\mathcal{D}a(t)+w(t))^3\quad&\text{for}\;\;t\in (t_1,t_2),\\
    a_t(t)=(1-\frac{\pi^2}{4})a(t)+G(w(t),a(t))\quad&\text{for}\;\;t\in (t_1,t_2),\\
    w(t_1)=w_0,\;\;a(t_1)=a_0,
\end{array}\right.
\end{equation}
admits a unique solution $(w,a)\in C([t_1,t_2];\mathcal H)\cap L^2(t_1,t_2;\mathcal H^1)$ satisfying
\begin{equation}\label{es1}
\|(w,a)\|_{C([t_1,t_2];\mathcal{H})}\le 2r.
\end{equation}
Moreover, the solution can be extended to a unique maximal solution on
$[t_1,T_{\max})$. If $T_{\max}<+\infty$, then $\lim\limits_{t\to T_{\max}^-}\|(w(t),a(t))\|_{\mathcal H}=+\infty.$
\end{lemma}
\begin{proof}    
Let us define $X=\Bigl(C([t_1,t_1+T_0];L^2(0,1))\cap L^2(t_1,t_1+T_0;L^{\infty}(0,1))\Bigr)\times  C([t_1,t_1+T_0])$, where $T_0$ will be chosen later. For $\mu>0$, we introduce the ${\mu}$-norm on $X$ by
$$\|(w,a)\|_{{\mu}}:=\|w\|_{C([t_1,t_1+T_0];L^2(0,1))}+\frac{1}{\mu}\|w\|_{L^2(t_1,t_1+T_0;L^{\infty}(0,1))}+\|a\|_{ C([t_1,t_1+T_0])},\;\;\forall (w,a)\in X.$$
We consider the mapping $\Gamma:X\to X$, where $\Gamma((w,a))$ is the unique mild solution $(z,b)$ of
\begin{equation}\label{fix1}
\left\{\begin{array}{ll}
    z_t(t)+Az(t)=w(t)-\mathcal{D}(G(w(t),a(t)))-(\mathcal{D}a(t)+w(t))^3\quad&\text{for}\;\;t\in (t_1,t_1+T_0),\\
    b_t(t)=(1-\frac{\pi^2}{4})a(t)+G(w(t),a(t))\quad&\text{for}\;\;t\in (t_1,t_1+T_0),\\
    z(t_1)=w_0,\;\;b(t_1)=a_0.
\end{array}\right.
\end{equation}
    
 Choose the ball $B_{2r}:=\{(w,a)\in X:\|(w,a)\|_{\mu}\le 2r\}$. According to Agmon's inequality, there exists $C>0$ such that for every $z\in C([t_1,t_1+T_0];L^2(0,1))\cap L^2(t_1,t_1+T_0;H^1_0(0,1))$,
 \begin{equation*}
    \|z\|_{L^2(t_1,t_1+T_0;L^{\infty}(0,1))}\le CT_0^{\frac{1}{4}}\|z\|^{1/2}_{L^2(t_1,t_1+T_0;H^1_0(0,1))}\|z\|^{1/2}_{C([t_1,t_1+T_0];L^2(0,1))}.
\end{equation*}
Combining this estimate with Lemma \ref{well-po-for-in-h}, and following the fixed-point argument in \cite{CX}, we use the nonlinear estimate
\begin{equation}\label{f-est-69}\|(\mathcal{D}a+w)^3\|_{L^1(t_1,t_1+T_0;L^{2}(0,1))}\le \|\mathcal{D}a+w\|_{L^{\infty}(t_1,t_1+T_0;L^2(0,1))}\|\mathcal{D}a+w\|_{L^2(t_1,t_1+T_0;L^{\infty}(0,1))}^2\end{equation}
together with the corresponding Lipschitz estimate for the difference of the nonlinear terms, to show that, provided $\mu>0$ and $T_0>0$ are sufficiently small, the mapping $\Gamma$ maps $B_{2r}$ into itself and is a contraction on $B_{2r}$. By the Banach fixed-point theorem, there exists a unique solution in $B_{2r}$. Finally, $w\in L^2(t_1,t_2;H^1_0(0,1))$ follows from Lemma \ref{well-po-for-in-h}.
\end{proof}
\begin{lemma}\label{w1}
Let $L>0$ and let $g\in L^{\infty}(\mathbb{R}^+) $ such that $\|g\|_{ L^{\infty}(\mathbb{R}^+) }\le L$. Then, for every $r>0$, there exists $T_2=T_2\left(r,L\right)>0$ such that for every $0\le t_1<t_2\le t_1+T_2$, and for every $(w_0,a_0)\in \mathcal{H}$ satisfying $\|(w_0,a_0)\|_{\mathcal{H}}\le r$, the system
\begin{equation}\label{we2}
 \left\{\begin{array}{ll}
    w_t(t)+Aw(t)=w(t)-\mathcal{D}(g(t))-(\mathcal{D}a(t)+w(t))^3\quad&\text{for}\;\;t\in (t_1,t_2),\\
    a_t(t)=(1-\frac{\pi^2}{4})a(t)+g(t)\quad&\text{for}\;\;t\in (t_1,t_2),\\
    w(t_1)=w_0,\;\;a(t_1)=a_0,
\end{array}\right.
\end{equation}
admits a unique solution $(w,a)\in C([t_1,t_2];\mathcal H)\cap L^2(t_1,t_2;\mathcal H^1)$ satisfying
\begin{equation}\label{es2}
\|(w,a)\|_{L^{\infty}(t_1,t_2;\mathcal{H})}\le 2 r.
\end{equation}
Moreover, the solution can be extended to a unique maximal solution on
$[t_1,T_{\max}')$. If $T_{\max}'<+\infty$, then $\lim\limits_{t\to T_{\max}'^-}\|(w(t),a(t))\|_{\mathcal H}=+\infty.$
\end{lemma}

\section{\large{Proof of the null controllability with control costs; Corollary \ref{main-th-null-control}}}
This proof strategy, which uses stabilization to obtain null controllability, was first introduced in \cite{CN} and later extended in \cite{CX, X} to obtain sharp control costs.
According to Remark \ref{quant-re}, we set
$$\rho_{\lambda}=e^{-D_p\lambda^{\frac{p}{p+1}}}, \quad M_{\lambda}=e^{D_p\lambda^{\frac{p}{p+1}}}.$$ 

Define the following quantities:
\begin{equation*}\left\{  \begin{aligned}
& {T}_0 = 0,\;\;{T}_n = {T}_{n-1 } + \frac{T}{2^{n }},\;\;\forall n \in \mathbb{N}^+,\\
&Q=2^{p+2}D_p\quad\text{and}\quad Q_p=\frac{Q^{p+1}(2^{p+1}-1)}{2^{p+2}(2^p-1)},\\
&R_p= e^{-\frac {Q_p}{T^p}}\quad\text{and}\quad\ell_n = \frac{Q^{p+1}  2^{(p+1)n}}{T^{p+1}},\;\;\forall n \in \mathbb{N}. \end{aligned}\right.
\end{equation*}
For $t\in({T}_{n}, {T}_{n+1}]$, we construct the feedback control \eqref{con-heat} corresponding to $\ell_n$, that is,
\begin{equation}\label{c-1}
u(t)=U_p(t,(y(t),a(t))):=\mathcal{K}_{\ell_n}(y(t),a(t)),\;\;\;\forall t\in (T_n,T_{n+1}].
\end{equation}

Note that $R_p<  e^{-D_p\ell_0^{\frac{p}{p+1}}}=\rho_{\ell_0}$. Theorem \ref{qu-ra-sta} gives rapid stabilization on $({T}_{0}, {T}_{1}]$. Hence, $(y,a)\in C([0,T_1];\mathcal H)$ and
\begin{equation*}
\|(y(T_1),a(T_1))\|_{\mathcal{H}}\le e^{D_p\ell_0^{\frac{p}{p+1}}-\ell_0\frac{T}{2}} \|(y_0,a_0)\|_{\mathcal{H}}\le e^{\frac{Q^{p+1}}{T^p}\frac{1-2^{p+1}}{2^{p+2}}} \|(y_0,a_0)\|_{\mathcal{H}}=e^{-\frac{Q_p(2^p-1)}{T^p}} \|(y_0,a_0)\|_{\mathcal{H}}.
\end{equation*}
Similarly, we obtain that
$$\|(y(T_n),a(T_n))\|_{\mathcal{H}}\le \prod_{k=0}^{n-1}e^{-2^{kp}\frac{Q_p(2^p-1)}{T^p}}\|(y_0,a_0)\|_{\mathcal{H}}=e^{-\frac{Q_p(2^{np}-1)}{T^p}}\|(y_0,a_0)\|_{\mathcal{H}},\quad\forall n\in \mathbb{N}.$$
Hence, $\|(y(T_n),a(T_n))\|_{\mathcal{H}}$ satisfies
\begin{equation}\label{yt69000}
\|(y(T_n),a(T_n))\|_{\mathcal{H}}\le e^{-\frac{Q_p(2^{np}-1)}{T^p}}R_p= e^{-2^{np}\frac{Q_p}{T^p}}< e^{-D_p{\ell_n}^{\frac{p}{p+1}}}=\rho_{\ell_n},\quad\forall n\in \mathbb{N},\end{equation}
which ensures the validity of the exponential decay estimates above and $(y,a)\in C([T_n,T_{n+1}];\mathcal H)$. Therefore, we have $(y,a)\in C([0,T);\mathcal H)$

Moreover, for $t\in (T_n,T_{n+1}],n\ge 1$, we have
\begin{equation}\label{yt690002}\|(y(t),a(t))\|_{\mathcal{H}}\le
e^{D_p\ell_n^{\frac{p}{p+1}}}e^{-\frac{Q_p(2^{np}-1)}{T^p}}
\|(y_0,a_0)\|_{\mathcal{H}}\le e^{-\frac{Q_p(2^{(n-1)p}-1)}{T^p}}\|(y_0,a_0)\|_{\mathcal{H}}.\end{equation}
Letting $n\to\infty$ and using $\lim\limits_{n\to+\infty}T_n=T$, we conclude from \eqref{yt69000} and \eqref{yt690002} that $\lim\limits_{t\to T^-}\|(y(t),a(t))\|_{\mathcal{H}}=0$. 

We now estimate the corresponding control cost. For every $n\ge 1$ and $t\in(T_{n},T_{n+1}]$, the quantitative bound on the feedback law gives
$$|u(t)|\le e^{D_p\ell_{n}^{\frac{p}{p+1}}}\|(y(T_{n}),a(T_{n}))\|_{\mathcal{H}}\le e^{-\frac{Q_p(2^{(n-1)p}-1)}{T^p}}\|(y_0,a_0)\|_{\mathcal{H}}\le e^{\frac{Q_p}{T^p}}\|(y_0,a_0)\|_{\mathcal{H}}.$$
Similarly, for $t\in(T_0,T_1]$,
$$|u(t)|\le e^{D_p\ell_{0}^{\frac{p}{p+1}}}\|(y_0,a_0)\|_{\mathcal{H}}\le e^{\frac{Q^{p+1}}{2^{p+2}T^p}}\|(y_0,a_0)\|_{\mathcal{H}}\le e^{\frac{Q_p}{T^p}}\|(y_0,a_0)\|_{\mathcal{H}}.$$
This completes the proof. $\hfill\square$

\section{\large{Proof of finite-time stabilization; Corollary \ref{f-t-s}}}

We emphasize that the uniform stability condition constitutes the essential distinction between finite-time null controllability and finite-time stabilization. With this additional condition, the control constructed in Corollary \ref{main-th-null-control} is no longer valid. This is because, if we set $t'=T_n$, then $\|\Upsilon(t,T_n;(y_0,a_0))\|_{\mathcal{H}}$ is bounded by $e^{D\sqrt{\ell_n}\ln \ell_n}\|(y_0,a_0)\|_{\mathcal{H}}$. Consequently, for sufficiently large $n$, we cannot obtain uniform stability at $t=T_{n+1}$. To overcome this difficulty, we employ a \textit{cut-off technique}. \\[-2mm] 

For $r>0$, we introduce a cut-off function $\chi_r\in C^{\infty}(\mathbb{R};[0,1])$ such that
$$\chi_r(x)=1\;\;\text{for }\;x\in [0,r],\quad\;\;\;\chi_r(x)=0\;\;\text{for }\;x\in [2r,+\infty).$$
Furthermore, we define an operator $\mathcal{F}_r:\mathbb{R}\to\mathbb{R}$ by
$$\mathcal{F}_r (u)=u\chi_r(|u|),\;\;\forall u\in\mathbb{R}.$$
$\mathcal{F}_r$ is Lipschitz with Lipschitz constant $1+2r\|\chi_r'\|_{C([0,2r])}$.

By Remark \ref{quant-re}, taking $p=2$ therein, we may choose a constant $D'\ge 1$ such that, for every $\ell>2\pi^2$,
$$\rho_{\ell}=e^{-D'\ell^{\frac{2}{3}}}\le 1/2, \;\;\; M_{\ell}=e^{D'\ell^{\frac{2}{3}}},\;\;\; \|\mathcal K_{\ell}\|_{\mathcal L}\le e^{{D}'\ell^{\frac{2}{3}}}.$$ 
Recalling the control designed in Corollary \ref{main-th-null-control}, we instead use the following control:
\begin{equation}\label{c-2}
u(t)=U(t,(y(t),a(t))):=\mathcal{F}_{\rho_{\ell_n}}(\mathcal{K}_{\ell_n}(y(t),a(t))),\;\;\;\forall t\in (T_n,T_{n+1}],
\end{equation}
and extend it periodically with the period $T$.

Then, on each interval $t\in (T_n,T_{n+1}]$, we have
\begin{equation}\label{F-boun}
|\mathcal{F}_{\rho_{\ell_n}}(\mathcal{K}_{\ell_n}(y,a))|\le \min\{1,\sqrt{2\|(y,a)\|_{\mathcal{H}}}\}.
\end{equation}
Indeed, if $|\mathcal{K}_{\ell_n}(y,a)|\le 2\rho_{\ell_n}$, the operator norm $\|\mathcal{K}_{\ell_n}\|_{\mathcal{L}}\le \rho_{\ell_n}^{-1}$ gives
$$|\mathcal{F}_{\rho_{\ell_n}}(\mathcal{K}_{\ell_n}(y,a))|\le|\mathcal{K}_{\ell_n}(y,a)|\le \sqrt{2\rho_{\ell_n}|\mathcal{K}_{\ell_n}(y,a)|}\le\sqrt{2\|(y,a)\|_{\mathcal{H}}}.$$
Moreover,
$|\mathcal{F}_{\rho_{\ell_n}}(\mathcal{K}_{\ell_n}(y,a))|\le|\mathcal{K}_{\ell_n}(y,a)|\le {2\rho_{\ell_n}}\le 1$.
If $|\mathcal{K}_{\ell_n}(y,a)|> 2\rho_{\ell_n}$, then by definition, we obtain that $\mathcal{F}_{\rho_{\ell_n}}(\mathcal{K}_{\ell_n}(y,a))=0$. 
\\[-1mm]

\noindent \textit{Proof of Corollary \ref{f-t-s}.} The proof proceeds in three steps. First, we prove the quantitative rapid stabilization with the nonlinear feedback law $u(t)=\mathcal{F}_{\rho_{\ell}}(\mathcal{K}_{\ell}(y(t),a(t)))$. Next, as in the proof of Corollary \ref{main-th-null-control}, we show that the system is locally null controllable with piecewise nonlinear feedback laws. Finally, we prove finite-time stabilization using the well-posedness property.
\hfill \\[-2mm] 

\noindent \textbf{Step 1. Quantitative rapid stabilization with $u(t)=\mathcal{F}_{\rho_{\ell}}(\mathcal{K}_{\ell}(y(t),a(t)))$.}

\begin{proposition}\label{qu-ra-sta-ne}
Let $\ell>2\pi^2$. For all $(y_0,a_0)\in \mathcal{H}$ satisfying $\|(y_0,a_0)\|_{\mathcal{H}}<\rho_{\ell}^3=e^{-3D'{\ell}^{\frac{2}{3}}}$, the Cauchy problem \eqref{heat-sys} with $u(t)=\mathcal{F}_{\rho_{\ell}}(\mathcal{K}_{\ell}(y(t),a(t)))$
has a unique solution $(y,a)\in C([s,+\infty);\mathcal{H})$ satisfying 
\begin{equation}\label{exp-s-heat2}
|u(t)|+	\|(y(t),a(t))\|_{\mathcal{H}}\le e^{D'{\ell}^{2/3}}e^{-\ell (t-s)}\|(y_0,a_0)\|_{\mathcal{H}},\;\;\forall t\ge s.
\end{equation}
\end{proposition}

\begin{proof}
Define
\begin{equation}\label{Gfeed-new}
\mathcal G_\ell(w,a):=\mathcal F_{\rho_\ell}\mathcal K_\ell(w+\mathcal Da,a).
\end{equation}
Since $\mathcal{F}_{\rho_{\ell}}\mathcal{K}_{\ell}$ is Lipschitz continuous with Lipschitz constant $(1+2{\rho_{\ell}}\|\chi_{\rho_{\ell}}'\|_{C([0,2\rho_{\ell}])})\rho_{\ell}^{-1}$, $\mathcal G_\ell$ is also Lipschitz continuous. Hence, Lemma \ref{w2} implies that the system \eqref{we1}, with the feedback law $\mathcal G_\ell$, admits a unique solution $(w,a)$ on $[s,T_{\max})$.

Set $w=y-\mathcal D a$. It is straightforward to verify that, under this change of variables, systems \eqref{heat-sys} and \eqref{we1} are equivalent. Hence, \eqref{heat-sys} with the feedback law $\mathcal{F}_{\rho_{\ell}}\mathcal{K}_{\ell}$ also admits a unique solution $(y,a)$ on $[s,T_{\max})$. We then show that the solution satisfies
\begin{equation}\label{eq}
\mathcal{F}_{\rho_{\ell}}(\mathcal{K}_{\ell}(y(t),a(t)))=\mathcal{K}_{\ell}(y(t),a(t)),\;\;\forall  t\in [s,T_{\max}).
\end{equation}
Fix $(y_0,a_0)\in \mathcal{H}$ such that
$\| (y_0,a_0)\|_{\mathcal{H}}<\rho_{\ell}^3$. We first claim that $\| (y(t),a(t))\|_{\mathcal{H}}< \rho_{\ell}^2$ for all $t\in [s,T_{\max})$.
	
We proceed by contradiction. Suppose there exists $s^*\in (s,T_{\max})$ such that $\|(y(s^*),a(s^*))\|_{\mathcal{H}}\ge \rho_{\ell}^2$.
Since $\| (y_0,a_0)\|_{\mathcal{H}}<\rho_{\ell}^3<\rho_{\ell}^2$ and $(y,a)\in C([s,T_{\max});\mathcal{H})$, we may define the shortest time $s'\in(s,s^*]$ such that $\| (y(t),a(t))\|_{\mathcal{H}}\ge \rho_{\ell}^2$. Accordingly,
$\| (y(t),a(t))\|_{\mathcal{H}}<  \rho_{\ell}^2$ for $t\in [s,s')$ and $\| (y(s'),a(s'))\|_{\mathcal{H}}=\rho_{\ell}^2$. Since $\|\mathcal K_\ell\|_{\mathcal L}\le\rho_\ell^{-1}$ and $\|(y(t),a(t))\|_{\mathcal H}<\rho_\ell^2$ for $t\in[s,s')$, we have $|\mathcal K_\ell(y(t),a(t))|<\rho_\ell$ for $t\in [s,s')$. Hence, the relation in \eqref{eq} holds on $[s,s')$, which yields the following exponential stability:
$$\| (y(t),a(t))\|_{\mathcal{H}}\le e^{D'{\ell}^{\frac{2}{3}}}\| (y_0,a_0)\|_{\mathcal{H}}e^{-\ell(t-s)}=\rho_{\ell}^{-1}\| (y_0,a_0)\|_{\mathcal{H}}e^{-\ell(t-s)}<\rho_{\ell}^2, \;\; \forall t\in [s,s').$$
Letting $t\to (s')^-$, we conclude $\| (y(s'),a(s'))\|_{\mathcal{H}}<\rho_{\ell}^2$, which contradicts the definition of $s'$. 

Consequently,
$$|\mathcal K_\ell(y(t),a(t))|
\le \rho_\ell^{-1}\|(y(t),a(t))\|_{\mathcal H}
<\rho_\ell,\quad \forall t\in[s,T_{\max}),$$
so the cut-off is inactive. Therefore,
$$|u(t)|+\|(y(t),a(t))\|_{\mathcal H}
\le e^{D'\ell^{2/3}}e^{-\ell(t-s)}
\|(y_0,a_0)\|_{\mathcal H},\quad \forall t\in[s,T_{\max}).$$
Since $w=y-\mathcal Da$, the norm equivalence $\|(w,a)\|_{\mathcal H}\le 2\|(y,a)\|_{\mathcal H}$, together with the bound above, shows that $(w,a)$ remains bounded on every interval $[s,T_{\max})$. Therefore, the blow-up alternative in Lemma \ref{w2} implies $T_{\max}=+\infty$. Thus \eqref{exp-s-heat2} holds.
\end{proof}
\noindent \textbf{Step 2. Finite-time null controllability with $u(t)=\mathcal{F}_{\rho_{\ell_n}}(\mathcal{K}_{\ell_n}(y(t),a(t))),\;\;\forall t\in(T_n,T_{n+1}]$.}

\begin{proposition}\label{main-th-null-control-ne}
There exists a constant $\tilde{Q} \geq 1 $ such that for every $T \in (0,1)$ and for every $(y_0,a_0)\in\mathcal{H}$ satisfying $\|(y_0,a_0)\|_{\mathcal{H}}\leq R:=e^{-\frac{\tilde{Q}}{T^2}}$, we have $\Upsilon(T,0;(y_0,a_0))=0$ in $\mathcal{H}$.
\end{proposition}
\begin{proof}
 Let us define the following quantities:
\begin{equation*}
\left\{  \begin{aligned}
&Q=30D',\quad R = e^{-\frac{7Q^3}{45T^2}},\\
&\ell_n = \frac{Q^3  2^{3n}}{T^3},\;\;\forall n \in \mathbb{N}.            
\end{aligned}\right.
\end{equation*}
Using the estimate in Proposition \ref{qu-ra-sta-ne} on each interval, the solution satisfies
$$\|(y(T_n),a(T_n))\|_{\mathcal{H}}\le \prod_{k=0}^{n-1} e^{-4^k\frac{7Q^3}{15T^2}}  \|(y_0,a_0)\|_{\mathcal{H}}\le e^{-(4^n-1)\frac{7Q^3}{45T^2}}R<\rho_{\ell_{n}}^3,\;\;\forall n\in \mathbb{N}.$$
Set $\tilde{Q}=\frac{7Q^3}{45}$. The rest of the
argument is carried out exactly as in the proof of Corollary \ref{main-th-null-control}.
\end{proof}
\noindent \textbf{Step 3. Finite-time stabilization.}

\noindent \textbf{Substep 1. Global (in time) well-posedness.}

The analysis of global well-posedness follows the argument in Coron-Xiang \cite{CX}. It is enough to establish well-posedness on $[0,T]$, since the feedback law is $T$-periodic. Using \eqref{Gfeed-new} and the Lipschitz continuity of  $\mathcal{F}_{\rho_{\ell_n}}\mathcal{K}_{\ell_n}$ on each interval $[T_n,T_{n+1}]$, Lemma \ref{w2} yields a unique maximal solution on each such interval. We now derive an a priori estimate that is uniform with respect to $n$.

The bound $|u|\le1$ and the equation $a_t=(1-\frac{\pi^2}{4})a+u$ imply that $a$ and $a_t$ remain bounded on finite time intervals. Let $w=y-\mathcal D a$. On every interval on which the solution exists, $w$ satisfies
$$ w_t+Aw=w-\mathcal Du-y^3.$$
Taking the $L^2(0,1)$ inner product with $w$ gives
\begin{equation}\label{global-L2-energy222} 
\begin{aligned} 
\frac12\frac{d}{dt}\|w(t)\|_{L^2}^2 +\|w_x(t)\|_{L^2}^2 &= \|w(t)\|_{L^2}^2 -u(t)\langle\psi,w(t)\rangle -\|y(t)\|_{L^4}^4 +a(t)\int_0^1\psi(x)y^3(t,x)dx . 
\end{aligned} 
\end{equation}
Using $|u|\le 1$ and Young's inequality, we obtain
$$|u\langle\psi,w\rangle| \le \frac12\|w\|_{L^2}^2 +\frac12\|\psi\|_{L^2}^2\quad\; \text{and}\quad\; \left| a\int_0^1\psi y^3\,dx \right| \le |a|\|\psi\|_{L^4}\|y\|_{L^4}^3 \le \frac12\|y\|_{L^4}^4 +C_\psi |a|^4.$$
Consequently, there exists $C'>0$ such that
\begin{equation}\label{global-uniform-energy} 
\frac{d}{dt}\|w(t)\|_{L^2}^2 +2\|w_x(t)\|_{L^2}^2 +\|y(t)\|_{L^4}^4 \le 3\|w(t)\|_{L^2}^2 +C'\bigl(1+|a(t)|^4\bigr). 
\end{equation}

Integrating \eqref{global-uniform-energy}, and using the boundedness of $a$ together with Gr\"onwall's inequality, we obtain 
\begin{equation}\label{global-uniform-bound} 
\begin{aligned} &\sup_{0\le t<T}\|w(t)\|_{L^2}^2 +\int_0^T\|w_x(t)\|_{L^2}^2\,dt <\infty.\end{aligned} 
\end{equation} 
Hence, the blow-up alternative in Lemma~\ref{w2} allows the solution to be extended successively to every interval $[T_n,T_{n+1}]$, and therefore to the whole interval $[0,T)$. Finally, since $y=w+\mathcal D a$, \eqref{global-uniform-bound} implies
\begin{equation}\label{ya-regular} 
(y,a)\in C([0,T);\mathcal H) \cap L^\infty(0,T;\mathcal H), \qquad y\in L^2(0,T;H^1(0,1)). 
\end{equation} 
It remains to show that 
\begin{equation}\label{we-po-U}
(y(T),a(T)):=\lim_{t\to T^{-}}(y(t),a(t)) \;\;\text{exists in}\;\;\mathcal{H}.
\end{equation}

Let $\mathcal{T}\in (0,T)$, and let $(y_{\mathcal{T}}^{\pm },a_{\mathcal{T}}^{\pm })$ be solutions of 
\begin{equation}\label{maxtowepo}
\left\{\begin{array}{ll}
(y_{\mathcal{T}}^{\pm })_t(t,x) -(y_{\mathcal{T}}^{\pm })_{xx}(t,x)=f&\text{for}\;\;(t,x)\in(\mathcal{T},T)\times(0,1),  \\
(a_{\mathcal{T}}^{\pm })_t(t)=(1-\frac{\pi^2}{4})a_{\mathcal{T}}^{\pm }(t)\pm1 & \text{for}\;\;t\in(\mathcal{T},T),  \\
y_{\mathcal{T}}^{\pm }(t,0)=a_{\mathcal{T}}^{\pm }(t),\;\; y_{\mathcal{T}}^{\pm }(t,1)=0& \text{for}\;\;t\in(\mathcal{T},T),  \\
(y_{\mathcal{T}}^{\pm }(\mathcal{T},x),a_{\mathcal{T}}^{\pm }(\mathcal{T}))=(y(\mathcal{T},x),a(\mathcal{T}))& \text{for}\;\;x\in(0,1).
\end{array}\right.
\end{equation}
By \eqref{f-est-69} and \eqref{ya-regular}, we have $f:=y-y^3\in L^1(0,T;L^2(0,1))$. Set $z_{\mathcal T}^{\pm}:=y_{\mathcal T}^{\pm}-(1-x)a_{\mathcal T}^{\pm}$. Then $z_{\mathcal T}^{\pm}$ satisfies
$$(z_{\mathcal T}^{\pm})_t+Az_{\mathcal T}^{\pm}=f-(1-x)(a_{\mathcal T}^{\pm})_t.$$
Since $(a_{\mathcal T}^{\pm})_t\in L^1(\mathcal T,T)$, the right-hand side belongs to $L^1(\mathcal T,T;L^2(0,1))$. Lemma \ref{well-po-for-in-h} therefore gives $z_{\mathcal T}^{\pm}\in C([\mathcal T,T];L^2(0,1))$, and hence $(y_{\mathcal{T}}^{\pm },a_{\mathcal{T}}^{\pm })\in C([\mathcal{T},T];\mathcal{H})$.

Let us define $(w_{\mathcal{T}},b_{\mathcal{T}}):=(y_{\mathcal{T}}^+-y_{\mathcal{T}}^-,a_{\mathcal{T}}^+-a_{\mathcal{T}}^-)$. In what follows, $C$ denotes a generic positive constant independent of $\mathcal T$. The explicit ODE $$(b_{\mathcal{T}})_t=(1-\frac{\pi^2}{4})b_{\mathcal T}+2\;\;\;\text{with}\;\;b_{\mathcal{T}}(\mathcal T)=0$$ implies that $\|b_{\mathcal{T}}\|_{C([\mathcal{T},T])}+\|(b_{\mathcal{T}})_t\|_{L^1(\mathcal{T},T)}\le C(T-\mathcal{T})$.
Define $z_{\mathcal T}=w_{\mathcal T}-(1-x)b_{\mathcal T}$. One readily obtains $$(z_{\mathcal{T}})_t+Az_{\mathcal{T}}=-(1-x)(b_{\mathcal T})_t \;\;\;\text{with} \;z_{\mathcal{T}}(\mathcal{T})=0 \;\text{in}\; L^2(0,1).$$ It follows from \eqref{norm-est-in-h} that $\|z_{\mathcal{T}}\|_{C([\mathcal{T},T];L^2(0,1))}\le \|(1-x)(b_{\mathcal{T}})_t\|_{L^1(\mathcal{T},T;L^2(0,1))}$. Hence,
\begin{equation*} 
\|w_{\mathcal{T}}\|_{C([\mathcal{T},T];L^2(0,1))}\le \|(1-x)(b_{\mathcal{T}})_t\|_{L^1(\mathcal{T},T;L^2(0,1))}+\|(1-x)b_{\mathcal{T}}\|_{C([\mathcal{T},T];L^2(0,1))}\le C(T-\mathcal{T}).
\end{equation*}

Given $\varepsilon>0$, choose $\mathcal T_1 \in (0,T)$ so close to $T$ that $C(T-\mathcal T_1) \le \varepsilon/4$.
Therefore, 
\begin{equation}\label{we-U-cha2}
\|(w_{\mathcal{T}_1},b_{\mathcal{T}_1})\|_{C([\mathcal{T}_1,T];\mathcal{H})}\le  \frac{\varepsilon}{4}.
\end{equation}
Moreover, from $b_{\mathcal{T}_1}(\cdot)\ge 0$ and the maximum principle (Lemma \ref{max-for-in-h}), we further have
$$y_{\mathcal{T}_1}^-(t,\cdot)\le  y(t,\cdot)\le  y_{\mathcal{T}_1}^+(t,\cdot),\;\;\; a_{\mathcal{T}_1}^-(t)\le  a(t)\le  a_{\mathcal{T}_1}^+(t)\;\;\;\text{for}\;t\in[\mathcal{T}_1,T),$$
which, together with \eqref{we-U-cha2}, implies that
\begin{equation}\label{comp2}
\|( y_{\mathcal{T}_1}^+-y,a_{\mathcal{T}_1}^+-a)\|_{C([\mathcal{T}_1,T);\mathcal{H})}\le \|( w_{\mathcal{T}_1},b_{\mathcal{T}_1})\|_{C([\mathcal{T}_1,T);\mathcal{H})}\le\frac{\varepsilon}{4}.
\end{equation}
Since $( y_{\mathcal{T}_1}^+, a_{\mathcal{T}_1}^+)\in C([\mathcal{T}_1,T];\mathcal{H})$, there exists $\tilde{\mathcal{T}}_1\in [\mathcal{T}_1,T)$ such that
\begin{equation}\label{comp3}
\|( y_{\mathcal{T}_1}^+(t)-y_{\mathcal{T}_1}^+(T),a_{\mathcal{T}_1}^+(t)-a_{\mathcal{T}_1}^+(T))\|_{\mathcal{H}}\le\frac{\varepsilon}{4},\;\;\forall t\in[\tilde{\mathcal{T}_1},T].
\end{equation}
From \eqref{comp2} and \eqref{comp3}, we derive
\begin{equation}\label{comp4}
\|( y(t)-y(t'),a(t)-a(t'))\|_{\mathcal{H}}\le\varepsilon,\;\;\forall t,t'\in[\tilde{\mathcal{T}_1},T).
\end{equation}
This implies that $(y(t),a(t))$ is Cauchy in $\mathcal{H}$ as $t\to T^{-}$, and hence \eqref{we-po-U} is valid, which completes the proof of properness.

\noindent \textbf{Substep 2. Finite-time null controllability.}

Let $\varepsilon>0$. Combining the norm equivalence $$\|(y,a)\|_{\mathcal H}\le 2\|(y-\mathcal Da,a)\|_{\mathcal H}, \quad \|(y-\mathcal Da,a)\|_{\mathcal H}\le 2\|(y,a)\|_{\mathcal H}, $$ with the control bound $|u|\le 1$, Lemma \ref{w1}, applied to $(w,a)=(y-\mathcal Da,a)$, yields some $\hat T=\hat T(\varepsilon)\in(0,T)$ such that
\begin{equation}\label{1to269}\|\Upsilon(t,t';(y_0,a_0))\|_{\mathcal{H}}\le \varepsilon,\;\;\;\;\forall \|(y_0,a_0)\|_{\mathcal{H}}\le\varepsilon/8,\;\forall t'\in [\hat{T},T),\;\forall t\in[t',T].\end{equation}

Observe that only finitely many feedback laws are involved in $[0,\hat T)$, say those corresponding to $\ell_0,\ldots,\ell_{\mathcal N}$, where ${\mathcal N}={\mathcal N}(\varepsilon)$. Choosing $\tilde R=\tilde R(\varepsilon)>0$ sufficiently small so that $ \left(\prod_{n=0}^{\mathcal N}M_{\ell_n}\right)\tilde R$ is smaller than both $\varepsilon/8$ and $\min\limits_{0\leq n\leq\mathcal N}\rho_{\ell_n}^3$, Proposition \ref{qu-ra-sta-ne} can be applied successively on each interval $[T_n,T_{n+1})\cap[0,\hat T)$. Hence, by a finite induction over every interval, there exists a constant $\varsigma=\varsigma(\varepsilon)<\min\{\tilde R,\varepsilon/8\}$ such that
\begin{equation}\label{1to2610}\|\Upsilon(t,t';(y_0,a_0))\|_{\mathcal H}\leq \varepsilon/8,
\quad\;\;\forall \|(y_0,a_0)\|_{\mathcal H}\leq \varsigma,\;\;\forall t'\in[0,\hat T),\;\;\forall t\in[t',\hat T].\end{equation}

We now apply \eqref{1to2610} with $\varepsilon=R$ and set $\varrho:=\varsigma(R)$. From \eqref{1to269} and \eqref{1to2610}, if $\|(y_0,a_0)\|_{\mathcal{H}}\le \varrho$, then 
$$\|\Upsilon(T,t';(y_0,a_0))\|_{\mathcal{H}}\le R,\;\forall t'\in [0,T].$$
Proposition \ref{main-th-null-control-ne} and the $T$-periodicity of the feedback law give
$$\|\Upsilon(2T,t';(y_0,a_0))\|_{\mathcal{H}}=0,\;\forall t'\in [0,T].$$
This completes the proof of this part.

\noindent \textbf{Substep 3. Uniform stability.}

Let $\epsilon>0$. We first consider the case $t'\in[0,T]$. For every $r>0$, let $\hat T_r:=\hat T(r)$ and $\varsigma(r)$ be the constants constructed in Substep 2. We first recall the following consequence of the previous arguments. 
If $\|(y_0,a_0)\|_{\mathcal H}\le \varsigma(r)$, then
\begin{equation}\label{uniform-small-on-one-period}
\|\Upsilon(t,t';(y_0,a_0))\|_{\mathcal H}\le r,
\quad \;\forall\, t'\in[0,T],\;\; \forall\, t\in[t',T].
\end{equation}
Indeed, if $t'\in[0,\hat T_r)$, then by the definition of $\varsigma(r)$ we have $\|\Upsilon(\hat T_r,t';(y_0,a_0))\|_{\mathcal H}\le \frac r8$. \eqref{1to269} gives
$$\|\Upsilon(t,t';(y_0,a_0))\|_{\mathcal H}\le r,
\quad \forall t\in[\hat T_r,T].$$
If $t'\in[\hat T_r,T]$, the same estimate follows directly from Lemma \ref{w1}, since
$\varsigma(r)<r/8$. This proves \eqref{uniform-small-on-one-period}.

Now set
$$\eta:=\varsigma\left(\epsilon\right),
\qquad
\delta:=\min\left\{\varrho,\varsigma\left(\eta\right)\right\}=\min\left\{\varrho,\varsigma\left(\varsigma(\epsilon)\right)\right\}.$$
Assume that $\|(y_0,a_0)\|_{\mathcal H}\le\delta$. Applying
\eqref{uniform-small-on-one-period} with $r=\eta$, we obtain
$$\|\Upsilon(t,t';(y_0,a_0))\|_{\mathcal H}\le \eta=\varsigma\left(\epsilon\right),
\quad\; \forall t'\in[0,T],\;\;\forall t\in [t',T].$$
Using the $T$-periodicity of the feedback and applying
\eqref{uniform-small-on-one-period} once more, now with $r=\epsilon$, yields
$$\|\Upsilon(t,t';(y_0,a_0))\|_{\mathcal H}\le \epsilon,\quad \forall t\in[T,2T].$$
Together with the first application, this gives
$$\|\Upsilon(t,t';(y_0,a_0))\|_{\mathcal H}\le \epsilon,\quad \forall t\in[t',2T].$$

On the other hand, since $\delta\le\varrho$, Substep 2 implies
$$\Upsilon(2T,t';(y_0,a_0))=0,\quad \forall t'\in[0,T].$$
It follows that $\Upsilon(t,t';(y_0,a_0))=0$ for all $ t\ge 2T$. Therefore,
$$\|\Upsilon(t,t';(y_0,a_0))\|_{\mathcal H}\le \epsilon,\quad \;\forall t'\in[0,T],\;\;\forall t\ge t'.$$
By $T$-periodicity, the same argument applies to any $t'\ge0$, completing the proof of uniform stability.\hfill$\square$
\end{appendices}

\bibliographystyle{siamplain}  
\bibliography{references}
\end{document}